\documentclass[11pt]{amsart}%
\usepackage{amsmath}%
\usepackage{amssymb}%
\usepackage{amscd}%
\usepackage{color}%
\usepackage{mathrsfs}%
\usepackage[all]{xy}%
\usepackage{newtxtext,newtxmath}%
\usepackage[all, warning]{onlyamsmath}
\usepackage[alphabetic]{amsrefs}
\usepackage{bbm}
\def\ol#1{\overline{#1}}
\def\wt#1{\widetilde{#1}}
\def\ul#1{\underline{#1}}
\theoremstyle{plain}
    \newtheorem{theorem}{Theorem}[section]
    
    \newtheorem{proposition}[theorem]{Proposition}
    \newtheorem{lemma}[theorem]{Lemma}
    \newtheorem{corollary}[theorem]{Corollary}
      
\theoremstyle{definition}
    \newtheorem{definition}[theorem]{Definition}
    
    \newtheorem{example}[theorem]{Example}
    \newtheorem{remark}[theorem]{Remark}
\def\Alphabet{A,B,C,D,E,F,G,H,I,J,K,L,M,N,O,P,Q,R,S,T,U,V,W,X,Y,Z}
\def\alphabet{a,b,c,d,e,f,g,h,i,j,k,l,m,n,o,p,q,r,s,t,u,v,w,x,y,z}
\def\endpiece{xxx}
\def\makeAlphabet[#1]{\expandafter\makeA#1,xxx,}
\def\makealphabet[#1]{\expandafter\makea#1,xxx,}
\def\makeA#1,{\def\temp{#1}\ifx\temp\endpiece\else%
\mkbb{#1}\mkfrak{#1}\mkbf{#1}\mkcal{#1}\mkscr{#1}\mkbs{#1}\expandafter\makeA\fi}%
\def\makea#1,{\def\temp{#1}\ifx\temp\endpiece\else\mkfrak{#1}\mkbf{#1}\mkbs{#1}\expandafter\makea\fi}%
\def\mkbb#1{\expandafter\def\csname bb#1\endcsname{\mathbb{#1}}}
\def\mkfrak#1{\expandafter\def\csname fr#1\endcsname{\mathfrak{#1}}}
\def\mkbf#1{\expandafter\def\csname b#1\endcsname{\mathbf{#1}}}
\def\mkcal#1{\expandafter\def\csname c#1\endcsname{\mathcal{#1}}}
\def\mkscr#1{\expandafter\def\csname s#1\endcsname{\mathscr{#1}}}
\def\mkbs#1{\expandafter\def\csname bs#1\endcsname{{\boldsymbol{#1}}}}
\def\makeop[#1]{\xmakeop#1,xxx,}
\def\mkop#1{\expandafter\def\csname #1\endcsname{{\mathrm{#1}}}} %
\def\xmakeop#1,{\def\temp{#1}\ifx\temp\endpiece\else\mkop{#1}\expandafter\xmakeop\fi}%
\def\makeup[#1]{\xmakeup#1,xxx,}
\def\mkup#1{\expandafter\def\csname #1\endcsname{{\mathrm{#1}\,}}} %
\def\xmakeup#1,{\def\temp{#1}\ifx\temp\endpiece\else\mkup{#1}\expandafter\xmakeup\fi}%
\makeAlphabet[\Alphabet]
\makealphabet[\alphabet]
\makeop[Alt,End,Ext,Hom,Sym,Tor,dR,Li,Cone,Tot,res,Res,id,Gr,MHS,Re,Vec,Rep,Rex,bil,GL,Lie,pr,Gd,OF,Dec]
\makeup[Spec,Ker,Coker,Im]

\def\cLog{\cL\!\operatorname{og}}
\def\sLog{\sL\!\!\operatorname{og}}
\def\bbLog{\bbL\!\operatorname{og}}

\def\bsmu{I}
\def\abs#1{|#1|}
\def\pol{\boldsymbol{\operatorname{pol}}}
\def\bone{\mathbbm{1}}
\def\bsxi{{\boldsymbol{\xi}}}
\numberwithin{equation}{section}

\begin{document}
\title{The Hodge realization of the polylogarithm on the product of multiplicative groups}
\author[Bannai]{Kenichi Bannai$^{*\diamond}$}
\email{bannai@math.keio.ac.jp}
\address{${}^*$Department of Mathematics, Faculty of Science and Technology, Keio University, 3-14-1 Hiyoshi, Kouhoku-ku, Yokohama 223-8522, Japan}
\address{${}^\diamond$Mathematical Science Team, RIKEN Center for Advanced Intelligence Project (AIP),1-4-1 Nihonbashi, Chuo-ku, Tokyo 103-0027, Japan}
\author[Hagihara]{Kei Hagihara$^{\diamond*}$}
\author[Yamada]{Kazuki Yamada$^*$}
\author[Yamamoto]{Shuji Yamamoto$^{*\diamond}$}

\date{\today}
\date{July 22, 2019\quad(Version 2.05)}
\begin{abstract}
	The purpose of this article is to describe explicitly the polylogarithm class in absolute 
	Hodge cohomology of a product of multiplicative groups,
	in terms of the Bloch-Wigner-Ramakrishnan polylogarithm functions.
	We will use the logarithmic Dolbeault complex defined by Burgos to calculate the corresponding absolute Hodge cohomology groups.
\end{abstract}
\thanks{This research was conducted as part of the KiPAS program FY2014--2018 of the Faculty of Science and Technology at Keio University.
This research was supported in part by KAKENHI 26247004, 16J01911, 16K13742, 18H05233 as well as the 
JSPS Core-to-Core program ``Foundation of a Global Research Cooperative
Center in Mathematics focused on Number Theory and Geometry''. }
\subjclass[2010]{14C30, 11G55} 
\maketitle
\setcounter{tocdepth}{1}
\setcounter{section}{0}

%
%
%
%
%
\section{Introduction}
%
%
%
%
%

The classical polylogarithm was constructed by Beilinson and Deligne as a class in the motivic cohomology of $\bbP^1\setminus\{0,1,\infty\}$.
The Hodge realization of this class gives an element in the corresponding absolute Hodge cohomology and may be described in terms of a
unipotent variation of mixed $\bbR$-Hodge structures on $\bbP^1\setminus\{0,1,\infty\}$ whose periods are given by
the classical polylogarithm functions (see \cite{HW98}).  
Since $\bbP^1 \setminus \{0,1,\infty\} = \bbG_{\mathrm{m}} \setminus \{1\}$,
we may view the classical polylogarithm as an object associated to the multiplicative group $\bbG_{\mathrm{m}}$. 
In fact, the construction of the polylogarithm was extended  to elliptic curves by Beilinson-Levin \cite{BL94},
to abelian schemes by Wildeshaus \cite{Wil97} and Kings \cite{Kin09}, and
more recently to general commutative group schemes by Huber and Kings \cite{HK18}.
The explicit description of the Hodge realization of the polylogarithm  was given
for elliptic curves by Beilinson-Levin \cite{BL94} (see also \cite{BKT10}*{Appendix}),
and the description of the topological sheaf underlying the Hodge realization of the polylogarithm
for general abelian varieties was given by Levin \cite{Lev97} and Blotti\`ere \cite{Blo09}.
The purpose of this article is to describe directly the Hodge realization of the polylogarithm 
as an explicit class in absolute Hodge cohomology for the product of multiplicative groups.

For an integer $g>0$, let $X = \bbG_{\mathrm{m}}^g$ be the $g$-fold product of the multiplicative group defined over 
the field of complex numbers $\bbC$.
We denote by $\bbLog$ the logarithm sheaf (see Definition \ref{def: log}),
which is a certain admissible pro-unipotent variation of mixed $\bbR$-Hodge structures on $X$.
The Hodge realizaton of the polylogarithm is defined as a class
\[
	\pol \in H^{2g-1}_\sA(U, \bbLog(g))
\]
in the absolute Hodge cohomology $H^{2g-1}_\sA(U, \bbLog(g))$ of $U:= X \setminus \{1\}$ with coefficients in $\bbLog(g)$,
the $g$-fold Tate twist of $\bbLog$.
The polylogarithm class is characterized as the class which maps to $1$
through the residue map.
In this article, we describe the polylogarithm class in terms of the Bloch-Wigner-Ramakrishnan polylogarithm functions
defined by Ramakrishnan \cite{Ram86} (see also \cite{Zag90}*{\S1}),
which are single-valued real analytic variants of the classical polylogarithm functions.

Contrary to the case when $g=1$, the classes in absolute Hodge cohomology of degree $>1$ cannot be interpreted in terms of
variations of mixed $\bbR$-Hodge structures on $U$.  We will use the logarithmic Dolbeault complex defined by Burgos \cite{Bur94}
to construct an explicit complex yielding absolute Hodge cohomology.
Then in our main result, Theorem \ref{main theorem},
we will describe the polylogarithm class in terms of cocycles.  
We admit that our result is not very surprising, but our method illustrates the strength of the logarithmic Dolbeault complex in 
elegantly describing directly the polylogarithm class in absolute Hodge cohomology for higher dimensional cases.

Our method also applies to the case $g=1$ and then recovers the following well-known result.

\begin{corollary}[=Corollary \ref{corollary: main}]
	Let $d>1$ be an integer and let $\zeta$ be a primitive $d$-th root of unity.  The inclusion $i_\zeta: \Spec \bbC \rightarrow U\coloneqq\Spec\bbC[t,t^{-1},(t-1)^{-1}]$
	defined by $t \mapsto \zeta$ induces the restriction map 
	\[
		i^*_\zeta: H^1_\sA(U, \bbLog(1)) \rightarrow H^1_\sA(\Spec \bbC, i^*_\zeta \bbLog(1)) \cong \prod_{k=0}^\infty \bbC/(2\pi i)^{k+1}\bbR.
	\]
	Then we have
	\[
		i^*_\zeta \pol  = \bigl( (-1)^k\Li_{k+1}(\zeta)\bigr)_{k=0}^\infty \in \prod_{k=0}^\infty \bbC/(2\pi i)^{k+1} \bbR.
	\]
\end{corollary}

Here, $\Li_k(t)$ denotes the classical polylogarithm function, defined as the complex analytic continuation of the 
series
\[ 
	\Li_k(t) := \sum_{n=1}^\infty \frac{t^n}{n^k}
\]
on the open unit disc in the complex plane.

The arithmetic significance of our result for the case $g>1$ will be treated in subsequent research.

\tableofcontents

The content of this article is as follows.  In \S \ref{section: MHS}, we will review the theory of mixed $\bbR$-Hodge structures
and the construction of mixed $\bbR$-Hodge structure on cohomology with coefficients in an admissible unipotent variation of mixed
$\bbR$-Hodge structures.   In \S \ref{section: polylogarithm}, we will define the logarithm sheaf $\bbLog$
on $X = \bbG_{\mathrm{m}}^g$, and describe the geometric
polylogarithm class  in terms of differential forms.
In \S \ref{section: log Dolbeault}, we will introduce the logarithmic Dolbeault complex defined by Burgos \cite{Bur94}, which will be used
to describe the absolute Hodge cohomology of $U= X \setminus \{1\}$.  In \S \ref{section: main theorem}, we will define 
the absolute polylogarithm class and prove our main result.

\subsection*{Acknowledgement} 
The authors express their sincere gratitude to the referee for carefully reading  the manuscript and
giving helpful suggestions.
The authors would also like to thank the KiPAS program FY2014--2018 of the Faculty of Science and Technology at 
Keio University for providing an excellent environment making this research possible.

%
%
%
%
%
\section{Mixed Hodge Structure on Cohomology}\label{section: MHS}
%
%
%
%
%

In what follows, for a smooth scheme $X$ of finite type over $\bbC$, we denote by $\cO_X$ the sheaf of holomorphic functions on $X(\bbC)$, and by $\Omega^\bullet_{X}$  
the complex of sheaves of holomorphic differential forms.
In this section, we review the construction by Deligne \cite{Del71} and Hain-Zucker \cite{HZ87} of mixed $\bbR$-Hodge structures on cohomology 
of a complex variety with coefficients in an admissible unipotent variation of mixed $\bbR$-Hodge structures.

%
%
\subsection{Construction of Mixed Hodge Structures on Cohomology}
%
%

We first start with the definition of pure and mixed $\bbR$-Hodge structures.

\begin{definition}[pure $\bbR$-Hodge structure]\label{def: pure}	
	Let $V = (V, F^\bullet)$ be a pair consisting of a finite dimensional $\bbR$-vector space and a 
	finite descending Hodge filtration by $\bbC$-linear subspaces $F^\bullet$ on $V_\bbC := V \otimes \bbC$.
	We say that $V$ is a \textit{pure $\bbR$-Hodge structure of weight $n$}, if
	\begin{equation*}
		V_\bbC=\bigoplus_{p+q=n}F^pV_\bbC\cap\ol F^qV_\bbC,
	\end{equation*}
	where $\ol F^qV_\bbC := \ol{F^qV}_\bbC$ is the complex conjugate of $F^q V_\bbC$.
\end{definition}

\begin{example}\label{example: Tate}
	For any integer $n\in \bbZ$, the Tate object $\bbR(n) := (V, F^\bullet)$, where $V_\bbC : = \bbC \ul\omega^n$ for a basis $\ul\omega^n$ 
	and $V := \bbR \ul u^n \subset V_\bbC$ for $\ul u^n := (2 \pi i)^n \ul\omega^n$, with a filtration $F^\bullet$ satisfying $F^{-n} V_\bbC = V_\bbC$ and 
	$F^{-n+1} V_\bbC = 0$ is a pure $\bbR$-Hodge structure of weight $-2n$.
\end{example}

\begin{definition}[mixed $\bbR$-Hodge structure]\label{def: MHS}
    Let $V=(V,W_\bullet,F^\bullet)$ be a triple consisting of a finite dimensional $\bbR$-vector space $V$, 
    a finite ascending filtration $W_\bullet$ on $V$ by $\bbR$-linear subspaces, and a finite descending 
    filtration $F^\bullet$ on $V_\bbC$ by $\bbC$-linear subspaces.
    We say that $V$ is a \textit{mixed $\bbR$-Hodge structure}, if  the pair
    $\Gr^W_n V := (\Gr^W_n V, F^\bullet)$ for each $n\in\bbZ$ is a pure $\bbR$-Hodge structure of weight $n$, where 
    $\Gr^W_nV:=W_nV/W_{n-1}V$ and $F^\bullet$ is the filtration on $\Gr^W_nV$ induced by $F^\bullet$.
    
    We denote by $\MHS_\bbR$ the category of mixed $\bbR$-Hodge structures, whose morphisms are $\bbR$-linear
    homomorphisms of the underlying $\bbR$-vector spaces compatible with the filtrations $W_\bullet$ and $F^\bullet$.
\end{definition}

Let $X$ be a smooth algebraic variety of dimension $g$ over $\bbC$, which we view as a complex manifold.  The purpose of this section is to give the 
construction of a mixed $\bbR$-Hodge structure on the cohomology $H^m(X,\bbR):= H^m(X(\bbC),\bbR)$.  
We first start with the construction of the 
weight filtration on $H^m(X,\bbR)$.
Let $j: X \hookrightarrow \ol X$ be a smooth and proper compactification of $X$ 
such that the complement $D: = \ol X \setminus X$ is a simple normal crossing divisor.  
Note that we have $H^m(X,\bbR) = \bbH^m(\ol X, R j_* \bbR)$.
For any complex $K^\bullet$ in an abelian category, let $\tau_{\leq\bullet}$ be the canonical filtration which is defined by
\[
	\tau_{\leq n}(K^m) =
	\begin{cases}
		0  &  n<m,\\
		\Ker(d^n) &  n=m,\\
		K^m  &  n>m.
	\end{cases}
\]
We define the filtration $\wt W_\bullet$ on $H^m(X,\bbR)$ to be the filtration induced from the canonical filtration $\tau_{\leq\bullet}$ on $R j_* \bbR$.
In other words, we let 
\[
	\wt W_n H^m(X,\bbR) := \Im\bigl(\bbH^m(\ol X, \tau_{\leq n}R j_* \bbR)\rightarrow H^m(X,\bbR)\bigr).
\]
We define the filtration $W_\bullet:= \wt W_\bullet[m]$ on $H^m(X,\bbR)$ by $W_n H^m(X,\bbR) := \wt W_{n-m} H^m(X,\bbR)$.

In order to define the Hodge filtration $F^\bullet$ on cohomology, we introduce the de Rham cohomology.
For the complex of differentials $\Omega^\bullet_X$ on $X$, we define the de Rham cohomology of $X$ by 
\[
	H^m_\dR(X):= \bbH^m(X, \Omega^\bullet_X).
\]
Note that we have canonical isomorphisms $H^m(X,\bbR)\otimes\bbC = H^m(X,\bbC) \cong H^m_\dR(X)$.
As in  \cite{Del71}*{(3.1.2)}, we let $\Omega^\bullet_{\ol X}(D)$ be the complex of sheaves of logarithmic differentials around $D$.
By \cite{Del71}*{(3.1.7)} (see Proposition \ref{prop: Deligne} below), we have an isomorphism
\begin{equation}\label{eq: log de Rham}
	H^m_\dR(X)\cong \bbH^m(\ol X,\Omega^\bullet_{\ol X}(D)).
\end{equation}

We define the weight and Hodge filtrations on de Rham cohomology.
For any $x \in \ol X$, we say that a coordinate neighborhood $U \subset \ol X$ of $x$ is \textit{adapted to $D$}, 
if $x$ has coordinates $(0,\ldots,0)$ and $U \cap D$ is defined by the equation $t_{\mu_1}\cdots t_{\mu_h}=0$ for local coordinates $t_1,\ldots,t_g$ on $U$.
For a fixed $U$ adapted to $D$, we let $\bsmu:=\{\mu_1,\ldots,\mu_h\}$.  
Then following \cite{Del71}*{(3.1.5)}, we define the weight filtration $W_\bullet$ on $\Omega^\bullet_{\ol X}(D)$ to be the multiplicative ascending filtration defined on 
each open neighborhood $U \subset \ol X$ adapted to $D$ by giving weight $0$ to the sections of $\Omega^\bullet_{\ol X}$ and weight $1$ 
to the sections $dt_\mu/t_\mu$ for $\mu\in\bsmu$.   The Hodge filtration $F^\bullet$ on $\Omega^\bullet_{\ol X}(D)$ is defined to be the stupid filtration
\[
	F^p \Omega^m_{\ol X}(D) :=
	\begin{cases}
		 \Omega^m_{\ol X}(D) &  p \leq m,\\
		0  & p>m.
	\end{cases}
\]
We define the filtrations $\wt W_\bullet$ and $F^\bullet$ on the de Rham cohomology $H^m_\dR(X)$ to be the filtrations
induced through the isomorphism \eqref{eq: log de Rham} from the filtrations $W_\bullet$ and $F^\bullet$ on $\Omega^\bullet_{\ol X}(D)$.  In other words, we let
\begin{align*}
	\wt W_n H^m_\dR(X) &:= \Im\bigl(\bbH^m(\ol X, W_n\Omega^\bullet_{\ol X}(D))\bigr)\rightarrow H^m_\dR(X)\bigr)\\
	F^p H^m_\dR(X) &:= \Im\bigl(\bbH^m(\ol X, F^p\Omega^\bullet_{\ol X}(D))\bigr)\rightarrow H^m_\dR(X)\bigr).
\end{align*}
By \cite{Del71}*{(3.1.7)}, we have the following result.

\begin{proposition}\label{prop: Deligne}
	The resolution $\bbC \hookrightarrow \Omega^\bullet_X$ and the quasi-isomorphism $\Omega^\bullet_{\ol X}(D)
	\hookrightarrow j_*\Omega^\bullet_X$ induce filtered quasi-isomorphisms
	\[
		 (R j_* \bbC, \tau_{\leq\bullet}) \rightarrow (j_* \Omega^\bullet_X, \tau_{\leq\bullet}) 
		\leftarrow (\Omega^\bullet_{\ol X}(D), \tau_{\leq\bullet}) \rightarrow 	(\Omega^\bullet_{\ol X}(D), W_\bullet),
	\]
	hence a quasi-isomorphism
	$
		\alpha:  (R j_* \bbC, \tau_{\leq\bullet})\xrightarrow\cong(\Omega^\bullet_{\ol X}(D), W_\bullet)
	$	
	in the filtered derived category.
\end{proposition}

Proposition \ref{prop: Deligne} implies that the filtration $\wt W_\bullet$ on $H^m_\dR(X)$ coincides with the filtration $\wt W_\bullet$ on $H^m(X,\bbR)$ through the isomorphism
$H^m(X,\bbR)\otimes\bbC\cong H^m_\dR(X)$.  Again, we denote by $W_\bullet$ the filtration $W_\bullet := \wt W_\bullet[m]$ on $H^m_\dR(X)$.
The filtrations $W_\bullet$ and $F^\bullet$ on $H^m_\dR(X)$ induces the filtrations $W_\bullet$ and $F^\bullet$ on $H^m(X,\bbR)\otimes\bbC$.

\begin{theorem}[\cite{Del71}*{Th\'eor\`eme 3.2.5}]\label{theorem: MHS}
	For any $m\in\bbZ$, the triple \[H^m(X,\bbR):=(H^m(X,\bbR), W_\bullet, F^\bullet)\] defined above is a mixed $\bbR$-Hodge structure.
\end{theorem}

%
%
\subsection{Admissible Unipotent Variation of Mixed Hodge Structures}
%
%

In this subsection, we review the definition of an admissible unipotent variation of mixed $\bbR$-Hodge structures on $X$,
and give the construction by Hain and Zucker of the mixed $\bbR$-Hodge structure on its cohomology. 

\begin{definition}
	Let $\bbV = (\bbV, W_\bullet, F^\bullet)$ be a triple, where $\bbV$ is an $\bbR$-local system on $X$,
	$W_\bullet$ is a finite ascending filtration of $\bbV$ by $\bbR$-local subsystems, and $F^\bullet$ is a finite descending filtration
	on $\cV := \bbV \otimes \cO_X$ by coherent $\cO_X$-submodules, satisfying the Griffiths transversality
	\[
		\nabla(F^p \cV) \subset F^{p-1}\cV\otimes\Omega^1_{X},
	\]
	where $\nabla$ is the integrable connection on $\cV$ defined by $\nabla:= 1\otimes d$.
	We say that $\bbV$ is a \textit{variation of mixed $\bbR$-Hodge structures} on $X$, 
	if for any point $x \in X$ giving an inclusion $i_x: \Spec \bbC \hookrightarrow X$, 
	the triple $i^*_x \bbV :=( i^*_x \bbV, W_\bullet, F^\bullet)$ is a mixed
	$\bbR$-Hodge structure, where $W_\bullet$ is the ascending filtration by $\bbR$-linear subspaces on 
	$i^*_x\bbV$ given by $W_n (i^*_x \bbV) :=i^*_x (W_n \bbV)$ for any $n\in\bbZ$, the filtration  
	$F^\bullet$ is the descending filtration by $\bbC$-linear subspaces on $i^*_x\bbV \otimes \bbC$ given by 
	$F^p(i^*_x \bbV \otimes\bbC) := i^*_x (F^p \cV) \otimes_{\cO_{x}} \bbC$ for any $p\in\bbZ$.
	\end{definition}

\begin{example}\label{example: constant}
	Suppose $V=(V,W_\bullet,F^\bullet)$ is a mixed $\bbR$-Hodge structure.  
	The constant variation of mixed $\bbR$-Hodge structure $V$ on $X$ is defined as the triple
	\[
		V_X :=(V_X, W_\bullet, F^\bullet),
	\]
	where $V_X$ and $W_\bullet$ are the constant $\bbR$-local systems on $X$ given by $V_X$ and $W_\bullet$, and the filtration
	$F^\bullet$ on $\cV:=V_X\otimes\cO_X$ is given by $F^p\cV := (F^p V_\bbC)_X \otimes\cO_X$.	
	In particular, if $V$ is a pure $\bbR$-Hodge structure of weight $n$, then 
	we say that $V_X$ is a constant variation of pure $\bbR$-Hodge structure of weight $n$.
	We will often denote $V_X$ simply by $V$ if there is no fear of confusion.
\end{example}

Example \ref{example: constant} shows that the Tate object $\bbR(n)$ of Example \ref{example: Tate} in $\MHS_\bbR$ 
defines a variation of pure $\bbR$-Hodge structure of weight $-2n$ on $X$.
We next define the notion of a unipotent variation of mixed $\bbR$-Hodge structures.

\begin{definition}
	Let $\bbV = (\bbV, W_\bullet, F^\bullet)$ be a variation of mixed $\bbR$-Hodge structures on $X$.  We say that $\bbV$ is 
	\textit{unipotent}, if $\Gr^W_n \bbV := (\Gr^W_n \bbV, F^\bullet)$ is a constant variation of pure $\bbR$-Hodge structure
	of weight $n$,  where $F^\bullet$ is the filtration on $\Gr^W_n \cV = \Gr^W_n \bbV \otimes \cO_X$ induced by $F^\bullet$.
\end{definition}

We next review the notion of admissibility of mixed Hodge structures defined in \cite{HZ87}*{(1.5)},
for the case of unipotent variation of mixed $\bbR$-Hodge structures.
Let $X \hookrightarrow \ol X$ be a smooth compactification of $X$ such that the complement $D:=\ol X\setminus X$ is a normal crossing divisor.
Let $\bbV = (\bbV, W_\bullet, F^\bullet)$ be a variation of mixed $\bbR$-Hodge structures on $X$.
We let $\cV := \bbV \otimes \cO_X$, with integrable connection $\nabla:=1\otimes d$.  By \cite{Del70}*{Proposition 5.2}, there exists
a canonical extension $\cV(D)$ of $\cV$, which is a certain locally free $\cO_{\ol X}$-module of finite type with integrable logarithmic connection along $D$.  
Since $W_n\cV := W_n\bbV\otimes \cO_X$ for $n\in\bbZ$ is an $\cO_X$-submodule of $\cV$ compatible with the connection, 
the canonical extension  of  $W_\bullet$ on $\cV$ defines a finite ascending
filtration $W_\bullet$ on $\cV(D)$.

\begin{definition}
	Let $\bbV = (\bbV, W_\bullet, F^\bullet)$ be a variation of mixed $\bbR$-Hodge structures on $X$. 
	Let $F^\bullet$ be a finite descending filtration on $\cV(D)$ extending the filtration $F^\bullet$ on $\cV$, again satisfying the Griffiths transversality with respect to $\nabla$.
	This implies in particular that $F^p\cV=F^p\cV(D)|_X$ for any $p\in\bbZ$.
	We say that $F^\bullet$ is \textit{compatible with $W_\bullet$}, if for any $n\in\bbZ$, the filtration $F^\bullet$ induced on $\Gr^{W}_n \cV(D)$
	extends the filtration $F^\bullet$ on $\Gr^W_n\cV$.
\end{definition}

The admissibility of $\bbV$ is defined as follows.

\begin{definition}\label{def: admissible}
	Suppose $\bbV$ is a unipotent variation of mixed $\bbR$-Hodge structures on $X$.	
	We say that $\bbV$ is \textit{admissible}, if it satisfies the following conditions.
	\begin{enumerate}
		\item Let $\cV(D)$ be the canonical extension of $\cV$ to $\ol X$.  
		The filtration $F^\bullet$ of $\cV$ extends to a filtration $F^\bullet$ of $\cV(D)$
		compatible with $W_\bullet$.		
		\item For any component $D_i$ of $D$, the residue $\Res_{D_i}$ 
		of $\nabla$ along $D_i$ defined in \cite{Del70}*{(3.8.3)} satisfies
		$\Res_{D_i}\bigl(W_n\cV(D)|_{D_i}\bigr) 
		\subset W_{n-2} \cV(D)|_{D_i}$ for any $n\in\bbZ$. 
	\end{enumerate}
\end{definition}

In what follows, let $\bbV=(\bbV,W_\bullet,F^\bullet)$ be an admissible unipotent variation of mixed $\bbR$-Hodge structures on $X$.
We let $\cV:= \bbV\otimes\cO_X$ with connection $\nabla:=1\otimes d$.
Then the de Rham cohomology of $X$ with coefficients in $\cV$ is defined
by
\[
	H^m_\dR(X, \cV) := \bbH^m(X, \cV\otimes\Omega^\bullet_{X}).
\]
Note that we have canonical isomorphisms $H^m(X, \bbV) \otimes \bbC =  H^m(X, \bbV_\bbC) \cong H^m_\dR(X, \cV)$.
If $\cV(D)$ is the canonical extension of $\cV$, then by \cite{Del70}*{Th\'eor\`eme 6.2} and the definition of the canonical extension, we have an isomorphism
\[
	H^m_\dR(X, \cV) \cong \bbH^m(\ol X, \cV(D)\otimes\Omega^\bullet_{\ol X}(D)).
\]
We define the filtrations $W_\bullet$ and $F^\bullet$ on $\cV(D)\otimes\Omega^\bullet_{\ol X}(D)$ by 
\[
	\begin{split}
		W_{n}(\cV(D)\otimes\Omega^\bullet_{\ol X}(D)) &:= \sum_{k\in\bbZ} W_{n-k} \cV(D)\otimes W_{k}\Omega^\bullet_{\ol X}(D), \\
		F^p(\cV(D)\otimes\Omega^\bullet_{\ol X}(D)) &:= \sum_{q\in\bbZ} F^{p-q} \cV(D)\otimes F^{q}\Omega^\bullet_{\ol X}(D).
	\end{split}
\]
The following result is due to Hain and Zucker.

\begin{proposition}[\cite{HZ87}*{(8.6) Proposition}]\label{prop: MHS}
	We denote by $\wt W_\bullet$ and $F^\bullet$ the filtrations induced on $H^m_\dR(X, \cV)$
	from the filtrations $W_\bullet$ and $F^\bullet$ of $\cV(D)\otimes\Omega^\bullet_{\ol X}(D)$, 
	and we denote again by $\wt W_\bullet$ the filtration on 
	$H^m(X, \bbV) \subset H^m_\dR(X, \cV)$ induced from $\wt W_\bullet$.   
	Then the triple 
	\[
		H^m(X,\bbV):=(H^m(X, \bbV), \wt W_\bullet[m], F^\bullet)
	\]
	 is a mixed $\bbR$-Hodge structure.
\end{proposition}

If $\bbV = (\bbV, W_\bullet, F^\bullet)$ is an admissible unipotent variation of mixed $\bbR$-Hodge structures on $X$, 
then we let $\bbV(n):= \bbV \otimes \bbR(n)$  for any $n\in\bbZ$, that is, $\bbV(n) =\bbV \otimes \bbR \ul u^n$ with the filtrations given by
$W_m(\bbV(n)) = (W_{m+2n} \bbV)\otimes\bbR\ul u^n$ and $F^p(\cV(n)) = (F^{p+n} \cV)\otimes\bbC\ul u^n$.
Then we have
$
	H^m(X, \bbV(n)) = H^m(X, \bbV)(n)
$
as mixed $\bbR$-Hodge structures.

%
%
%
%
%
\section{Geometric Polylogarithm}\label{section: polylogarithm}
%
%
%
%
%

In what follows, we let $g$ be a positive integer, and we let $X:=\bbG_{\mathrm{m}}^g=\Spec\bbC[t_1^{\pm1},\ldots,t_g^{\pm1}]$ 
be the $g$-fold product of the 
multiplicative group over $\bbC$.  The purpose of this section is to give the definition and an explicit construction
of the geometric polylogarithm class of $X$.
We start with the definition of the logarithm sheaf on $X$.

%
%
\subsection{The Logarithm Sheaf}
%
%

In this subsection, we define the logarithm sheaf $\bbLog$ of $X$.
In what follows, we let $\bbR(\bone): = \bbR(1)^{\oplus g}$, which is a pure $\bbR$-Hodge structure of weight $-2$.
We denote by $\bbN$ the set of natural numbers, and for any $\bsk = (k_\mu) \in \bbN^g$, we let $\abs{\bsk}\coloneqq 
k_1 + \cdots + k_g$.

For an integer $N>0$, consider  the free $\cO_{X}$-module
\[
	\cLog^N\coloneqq\prod_{\bsk \in\bbN^g, |\bsk|\leq N} \cO_{X} \ul\omega^\bsk
\]
with connection  
$
	\nabla\colon\cLog^N\rightarrow \cLog^N\otimes \Omega^1_{X}
$
given by
\begin{equation}\label{eq: connection}
	\nabla(\ul\omega^\bsk)= \sum_{\mu=1}^g \ul\omega^{\bsk+1_\mu} \otimes \frac{dt_\mu}{t_\mu},
\end{equation}
where $1_\mu$ denotes the element in $\bbN^g$ whose $\mu$-th component is $1$ and the other components are $0$,
and we let $\ul\omega^\bsk=0$ if $|\bsk|>N$.
We define the filtrations $W_\bullet$ and $F^\bullet$ on $\cLog^N$ by 
\begin{align*}
	W_{-2m} \cLog^N&\coloneqq W_{-2m+1} \cLog^N\coloneqq \prod_{m\leq |\bsk|\leq N}  \cO_{X} \ul\omega^\bsk,  &
	F^{-p} \cLog^N&\coloneqq\prod_{\abs{\bsk}\leq p}  \cO_{X} \ul\omega^\bsk.
\end{align*}
By choosing local branches of $\log t_\mu$, we define sections
\begin{equation}\label{eq: transformation}
	\ul u^\bsk\coloneqq(2\pi i)^{\abs{\bsk}}\exp\biggl(\sum_{\mu=1}^g (-\log t_\mu) \omega_\mu\biggr)\cdot\ul\omega^\bsk
\end{equation}
of $\cLog^N$,
where $\omega_\mu$ is an operator which acts on $\ul\omega^\bsk$ by 
$\omega_\mu\cdot\ul\omega^\bsk = \ul\omega^{\bsk + 1_\mu}$.  
We have
\begin{multline*}
	\nabla(\ul u^\bsk) =   (2\pi i)^{\abs{\bsk}}\exp\biggl(\sum_{\nu=1}^g(-\log t_\nu)\omega_\nu\biggr) \cdot \sum_{\mu=1}^g \ul \omega^{\bsk+1_\mu} \otimes \frac{dt_\mu}{t_\mu}\\
	 -  (2\pi i)^{\abs{\bsk}}\sum_{\mu=1}^g\exp\biggl(\sum_{\nu=1}^g(-\log t_\nu)\omega_\nu\biggr)\omega_\mu\cdot\ul \omega^\bsk\otimes \frac{dt_\mu}{t_\mu} =0,
\end{multline*}
hence $\ul u^\bsk$ gives a horizontal section of $\cLog^N$ with respect to the connection $\nabla$.
Since the transfomation \eqref{eq: transformation} is invertible, the sections
$\ul u^\bsk$ are linearly independent over $\bbC$, hence span the space of horizontal sections of $\cLog^N$.

We define the $\bbR$-local system $\bbLog^N$  by
\[
	\bbLog^N\coloneqq\prod_{\bsk\in\bbN^g, \abs{\bsk}\leq N}\bbR \ul u^\bsk,
\] 
which gives an $\bbR$-structure on the $\bbC$-local system of horizontal sections of $\cLog^N$.
Although the basis $\ul u^\bsk$ depends on the branches of $\log t_\nu$ and does not give a global basis of $\bbLog^N$,
the sheaf $\bbLog^N$ is independent of the branches of $\log t_\mu$.

We define the weight filtration on $\bbLog^N$ by 
\[
	W_{-2m} \bbLog^N= W_{-2m+1} \bbLog^N\coloneqq\prod_{m\leq\abs\bsk\leq N}  \bbR \ul u^\bsk,
\]
which is compatible with the weight filtration $W_\bullet$ on $\cLog^N$ through the natural isomorphism 
\[
	\bbLog^N\otimes\cO_X \cong \cLog^N.
\]
Then the triple $\bbLog^N=(\bbLog^N, W_\bullet, F^\bullet)$ is a variation of mixed $\bbR$-Hodge structures on $X$.
By mapping the sections 
$\ul u^\bsk$ of $\bbLog^N$ to $\ul u_1^{k_1} \cdots \ul u_g^{k_g}$, where $(\ul u_1, \ldots, \ul u_g)$ is the standard basis of $\bbR(\bone)$, we obtain an isomorphism
\[
	\Gr^W_\bullet \bbLog^N \cong \prod_{k=0}^N \Sym^k \bbR(\bone)
\]
of variations of $\bbR$-Hodge structures on $X$,
where we view $\Sym^k\bbR(\bone)$ as a constant variation of pure $\bbR$-Hodge structures on $X$.
We see from this fact that $\bbLog^N$ is a unipotent variation of mixed $\bbR$-Hodge structures.  

The unipotent variation of mixed $\bbR$-Hodge structues
$\bbLog^N$ is known to satisfy the \textit{splitting principle}, given as follows.
Let $d \geq 1$ and let $\zeta = (\zeta_1, \ldots, \zeta_g)$ be an element in $X$ whose components are $d$-th roots of unity,
and we denote by $i_\zeta: \Spec \bbC \hookrightarrow X$ the inclusion corresponding to $\zeta$.
By mapping the basis $\ul \omega^\bsk$ of $i^*_\zeta \cLog$ to $\ul \omega_1^{k_1} \cdots \ul \omega_g^{k_g}$,
we obtain an isomorphism
\begin{equation}\label{eq: splitting}
	i^*_\zeta \bbLog^N \cong \prod_{k=0}^N \Sym^k \bbR(\bone)
\end{equation}
of mixed $\bbR$-Hodge structures.

We let  $\ol X:= \bbP^1 \times \cdots \times \bbP^1$
be the $g$-fold product of $\bbP^1$.  We have a natural inclusion $X \hookrightarrow \ol X$ with complement $D: = \ol X \setminus X$.
We denote by $\Omega^\bullet_{\ol X}(D)$ the complex of differentials on $\ol X$ with logarithmic singularities along $D$.
Note that $\cLog^N$ extends naturally to a coherent $\cO_{\ol X}$-module
\[
	\cLog^N(D)\coloneqq\prod_{\abs\bsk\leq N} \cO_{\ol X} \ul\omega^\bsk,
\]
and \eqref{eq: connection} defines a connection 
\[
	\nabla\colon\cLog^N(D) \rightarrow \cLog^N(D) \otimes \Omega^1_{\ol X}(D)
\] 
with logarithmic singularities along $D$.
Then $\cLog^N(D)$ is the canonical extension of $\cLog^N$ to $\ol X$, and the filtrations
 $W_\bullet$ and $F^\bullet$ on $\cLog^N$ naturally extend to $\cLog^N(D)$.
We see from the existence of such $\cLog^N(D)$ that $\bbLog^N$ is an admissible unipotent variation of mixed $\bbR$-Hodge structures.

For integers $N > 0$, we have natural projections 
\begin{equation}\label{eq: projection}
	\bbLog^{N}\rightarrow\bbLog^{N-1}.
\end{equation}
This gives an exact sequence
\begin{equation}\label{eq: SES}
	0 \rightarrow \Sym^{N} \bbR(\bone) \rightarrow \bbLog^{N} \rightarrow \bbLog^{N-1} \rightarrow 0
\end{equation}
of variations of mixed $\bbR$-Hodge structures.  

\begin{definition}\label{def: log}
	We define the \emph{logarithm sheaf} $\bbLog$ to be the projective system with respect to \eqref{eq: projection}
	of admissible unipotent variations of mixed $\bbR$-Hodge structures $\bbLog^N$ on $X$.	
\end{definition}

We define the cohomology of $X$ with coefficients in $\bbLog$ by
\[
	H^m(X, \bbLog)\coloneqq \varprojlim_N H^m(X,  \bbLog^N)
\]
for any $m\in\bbZ$.  We will calculate this cohomology in the next subsection.

%
%
\subsection{Cohomology of the Logarithm Sheaf}
%
%

In \cite{HK18}*{Corollary 7.1.6}, Huber and Kings calculated the cohomology of the logarithm sheaf
for a general commutative group scheme 
given as an extension of an abelian variety by a torus.   Applied to $X$, this gives the following.

\begin{proposition}\label{prop: Log on X}
	We have
	\[
		H^m(X, \bbLog(g))\coloneqq\varprojlim_N H^m(X, \bbLog^N(g)) =
		\begin{cases}
			\bbR & m=g, \\ 
			0  & m\neq g.
		\end{cases}
	\]
\end{proposition}

Proposition \ref{prop: Log on X} follows immediately from the following Lemma \ref{lemma: Log on X}
concerning the cohomology of $X$ with coefficients in $\bbLog^N$,
a proof of which we give for the convenience of the reader.

\begin{lemma}\label{lemma: Log on X}
	For integers $m$ and $N>0$, the natural map
	\[
		H^m(X, \bbLog^N) \rightarrow H^m(X, \bbLog^{N-1})
	\]
	is the zero map if $m\neq g$, and is an isomorphism
	if $m=g$.
	Moreover, we have 
	\[
		H^g(X,\bbLog^N)\xrightarrow\cong\cdots\xrightarrow\cong H^g(X,\bbLog^0)\cong\bbR(-g).
	\]
\end{lemma}

\begin{proof}
	For $N>0$, the exact sequence \eqref{eq: SES} gives rise to the long exact sequence
	\begin{align}\label{eq: LES}\begin{split}
		0 &\rightarrow H^0(X,  \Sym^{N} \bbR(\bone)) \rightarrow H^0(X, \bbLog^{N}) \rightarrow H^0(X,  \bbLog^{N-1}) \\
		&\rightarrow H^1(X,  \Sym^{N} \bbR(\bone)) \rightarrow H^1(X, \bbLog^{N}) \rightarrow H^1(X,  \bbLog^{N-1})\\
		&\rightarrow \cdots \\
		\cdots & \rightarrow H^{g}(X, \Sym^{N} \bbR(\bone)) \rightarrow H^g(X, \bbLog^{N}) \rightarrow H^g(X,  \bbLog^{N-1})
		\rightarrow  0.
	\end{split}\end{align}
	In order to prove the statement for $m\neq g$,
	it suffices to show that the morphism 
	\begin{equation}\label{eq: surjection}
		H^m(X, \Sym^{N}\bbR(\bone)) \rightarrow H^m(X, \bbLog^{N})
	\end{equation}
	is surjective.
	We prove that the corresponding morphism for de Rham cohomology is surjective.
	We first let $\cH := \bbR(\bone) \otimes \cO_{\ol X}$ be the coherent $\cO_{\ol X}$-module on $\ol X$ with connection $\nabla := 1 \otimes d$.
	Let $\bbA^g := \Spec \bbC[t_1,\ldots,t_g] \subset \ol X= \bbP^1\times\cdots\times\bbP^1$, and we denote by $Y$ the normal crossing divisor of $\bbA^g$ defined by $t_1 \cdots t_g =0$.  
	Then the restriction to $\bbA^g$ of $\cH$ and $\cLog^N(D)$ give coherent $\cO_{\bbA^g}$-modules with logarithmic connection along $Y$, and we have
	\begin{align*}
		H^m_\dR(X, \Sym^N \cH) &\cong \bbH^m(\bbA^g, \Sym^N \cH \otimes \Omega^\bullet_{\bbA^g}(Y)) = H^m(\Gamma(\bbA^g, \Sym^N \cH \otimes \Omega^\bullet_{\bbA^g}(Y))) \\
		 H^m_\dR(X, \cLog^N)&\cong \bbH^m(\bbA^g, \cLog^N(D) \otimes \Omega^\bullet_{\bbA^g}(Y)) =  H^m(\Gamma(\bbA^g, \cLog^N(D) \otimes \Omega^\bullet_{\bbA^g}(Y))),
	\end{align*}
	where the equalities on the right follow from the fact that $\bbA^g$ is affine.
	This proves in particular that the de Rham cohomology is zero for $m > g$, hence it is sufficient to prove our assertion for non-negative integers $m<g$.
	Let $[\xi]$ be a class in $H^m_\dR(X, \cLog^N)$.   Since $\bbA^g$ is affine, the class $[\xi]$ is represented by a cocycle
	$\xi \in \Gamma(\bbA^g, \cLog^N(D) \otimes \Omega^m_{\bbA^g}(Y))$.  
	Since $\cLog^N(D) \cong \bigoplus_{k=0}^N \Sym^k \cH$ as $\cO_{\bbA^g}$-modules, we may write
	\[
		\xi = \sum_{k=0}^N \xi_k
	\]
	for $\xi_k \in \Gamma(\bbA^g, \Sym^k \cH \otimes \Omega^m_{\bbA^g}(Y))$.  Then the cocycle condition $\nabla(\xi) = 0$ gives
	\[
		\nabla(\xi) =  \sum_{k=0}^N \nabla(\xi_k)  = \sum_{k=0}^N  \left((1 \otimes d)(\xi_k)  
		 + \sum_{\mu=1}^{g} \omega_\mu\cdot\xi_k \wedge \frac{dt_\mu}{t_\mu}  \right) =0.
	\]
	This gives the differential equations
	\begin{align*}
		(1 \otimes d)(\xi_0) &= 0, &
		 (1 \otimes d)(\xi_{k}) & + \sum_{\mu=1}^{g}\omega_\mu\cdot\xi_{k-1} \wedge \frac{dt_\mu}{t_\mu}    =0
	\end{align*}
	for $k=1, \ldots, N$.  We prove by induction on $k$ that the class of $\xi$ may be represented by an element 
	satisfying $\xi_i = 0$ for $i < N$.   Let $k\geq 0$ be an integer $<N$ and suppose $\xi_i = 0$ for $i < k$.
	Then the above differential equation becomes
	\begin{align*}
		(1 \otimes d)(\xi_k) &= 0, &
		 (1 \otimes d)(\xi_{k+1}) & + \sum_{\mu=1}^{g} \omega_\mu  \cdot\xi_{k} \wedge \frac{dt_\mu}{t_\mu}    =0.
	\end{align*}
	The first differential equation shows that $\xi_k$ is a cocycle representing an element in 
	\[
		H^m(\Gamma(\bbA^g, \Sym^k \cH \otimes\Omega^\bullet_{\bbA^g}(Y))) \cong H^m_\dR(X, \Sym^k \cH).
	\]
	The second differential equation shows that if we write 
	\[
		\xi_k = \sum_{\boldsymbol\mu = \{ \mu_1, \cdots, \mu_m\} \subsetneq\{ 1, \cdots, g \}}  
		u_{\boldsymbol\mu}\frac{dt_{\mu_1}}{t_{\mu_1}} \wedge \cdots \wedge \frac{dt_{\mu_m}}{t_{\mu_m}}
	\]
	for $u_{\boldsymbol\mu} \in \Gamma(\bbA^g, \Sym^k \cH)$, the constant term of $u_{\boldsymbol\mu}$ is zero.  
	This shows that the class 
	\[
		[\xi_k]\in\bbH^m(\bbA^g, \Sym^k \cH\otimes\Omega^\bullet_{\bbA^g}(Y))
	\]
	is zero, 
	which implies that there exists an element $\eta_k \in \Gamma(\bbA^g,  \Sym^k \cH \otimes \Omega_{\bbA^g}^{m-1}(Y))$
	such that $(1 \otimes d)(\eta_k) = \xi_k$.  If we let 
	\[
		\wt\xi_{k+1}:= \xi_{k+1} - \sum_{\mu=1}^{g} \omega_\mu \cdot \eta_{k} \wedge \frac{dt_\mu}{t_\mu}
		\in\Gamma(\bbA^g,  \Sym^{k+1} \cH \otimes \Omega_{\bbA^g}^{m}(Y)),
	\]
	where $\omega_\mu$ acts on $\ul\omega^\bsm$ by $\omega_\mu\cdot\ul\omega^\bsm=\ul\omega^{\bsm+1_\mu}$,
	and $\wt\xi_{j} := \xi_j$ for $j= k+2, \ldots, N$, then the class of
	\[
		\wt\xi = \sum_{j= k+1}^N \wt \xi_j \in \Gamma( \bbA^g, \cLog^N \otimes \Omega^m_{\bbA^g}(Y))
	\]
	coincides with that of $\xi$.  By induction on $k$, we see that
	the class of $\xi$ coincides with the class of a certain element  $\xi_N  \in \Gamma( \bbA^g, \Sym^N \cH \otimes \Omega^m_{\bbA^g}(Y))$.  
	Since $\xi_N$ satisfies $(1 \otimes d)(\xi_N) = 0$, this element represents a class 
	\[
		[\xi_N] \in  H^m(\Gamma(\bbA^g, \Sym^{N}\cH \otimes\Omega^\bullet_{\bbA^g}(Y))) \cong H^m_\dR(X, \Sym^N \cH).
	\]
	By construction, $[\xi_N]$ maps to the class $[\xi]$ through the homomorphism \eqref{eq: surjection}.
	This gives our statement for the case $m\neq g$.	
	As a consequence,
	the exact sequence \eqref{eq: LES}
	gives an isomorphism
	\begin{equation}\label{eq: H0}
		H^0(X,  \Sym^{N} \bbR(\bone)) \xrightarrow\cong H^0(X, \bbLog^{N}),
	\end{equation}
	short exact sequences
	\begin{equation}\label{eq: Hm}
		0 \rightarrow H^{m-1}(X,  \bbLog^{N-1}) \rightarrow  H^{m}(X,  \Sym^{N} \bbR(\bone)) \rightarrow  H^{m}(X, \bbLog^{N})
		\rightarrow 0
	\end{equation}
	for $m=1, \ldots, g-1$,  and the exact sequence
	\begin{equation}\label{eq: aLES}
		0 \rightarrow H^{g-1}(X, \bbLog^{N-1}) \rightarrow H^g(X,  \Sym^N \bbR(\bone))
		\rightarrow H^g(X, \bbLog^N)
		\rightarrow H^g(X, \bbLog^{N-1}) \rightarrow 0.
	\end{equation}
	
	Next, we prove our assertion for the case $m=g$. 
	From \eqref{eq: aLES}, it is sufficient to prove that the 
	dimensions of 
	$H^{g-1}(X, \bbLog^{N-1})$ and $ H^g(X,  \Sym^N \bbR(\bone))$ coincide.
	Note that since $H^g(X,\bbR) =H^g(\bbG_{\mathrm{m}}^g,\bbR)\cong \bbR(-g)$, we have
	\[
		\dim_\bbR H^g(X,  \Sym^N \bbR(\bone)) = 
		\dim_\bbR  (\Sym^N \bbR(\bone))
		= \binom{N+g-1}{g-1}.
	\]
	We prove the equality
	\begin{equation}\label{eq: 9610}
		\dim_\bbR H^m(X, \bbLog^N) = \binom{N+g-1-m}{g-1-m} \binom{N+g}{m}
	\end{equation}
	for any integers $m=0,\ldots, g-1$ and $N \geq 0$.

	If $N=0$, then $\bbLog^0 $ is the constant $\bbR$-Hodge structure $\bbR(0)$ on $X$,
	hence
	\begin{equation}\label{eq: 321}
		H^m(X, \bbLog^0)  =  H^m(\bbG_{\mathrm{m}}^g,\bbR) = \bigwedge^m  H^1(\bbG_{\mathrm{m}}^g,\bbR) 
		\cong \bigwedge^m  \bbR(-1)^{\oplus g}  \cong \bbR(-m)^{\oplus\binom{g}{m}},
	\end{equation}
	which gives \eqref{eq: 9610} in this case.
	If $m=0$, then by \eqref{eq: H0}, we have
	\[
		\dim_\bbR H^0(X, \bbLog^{N}) = \dim_\bbR H^0(X,  \Sym^{N} \bbR(\bone)) 
		= \dim_\bbR \Sym^N \bbR(\bone) = \binom{N+g-1}{g-1},
	\]
	which gives \eqref{eq: 9610} in this case.
	Finally, suppose $m,N>0$ and \eqref{eq: 9610} is true for the pair $(m-1,N-1)$.  
	Then the exact sequence \eqref{eq: Hm} shows that we have
\begin{multline*}
	\dim_\bbR H^{m}(X, \bbLog^{N}) = \dim_\bbR H^{m}(X,  \Sym^{N} \bbR(\bone))  - \dim_\bbR H^{m-1}(X,  \bbLog^{N-1})  \\
	= \dim_\bbR \Biggl(\biggl(\bigwedge^m H^1(X) \biggr)\otimes 
	\Sym^{N} \bbR(\bone) \Biggr)  - \dim_\bbR H^{m-1}(X,  \bbLog^{N-1}) \\
	= \binom{g}{m} \binom{N+g-1}{g-1} - \binom{N+g-m-1}{g-m} \binom{N+g-1}{m-1} = \binom{N+g-m-1}{g-m-1} \binom{N+g}{m},
\end{multline*}
which proves \eqref{eq: 9610} for $(m,N)$ as desired.

The last statement of our lemma follows from \eqref{eq: 321} for the case $m=g$.
\end{proof}

%
%
\subsection{The Geometric Polylogarithm Class}
%
%

In this subsection, we give the definition of the geometric polylogarithm class of $X$.
We let $Z: = \{ 1\}\subset X$ and $U: = X \setminus Z$.  Then we have the following.

\begin{proposition}\label{prop: Log on U}
	For $g>1$, we have canonical isomorphisms
	\[
		H^{m}(U, \bbLog(g))\coloneqq \varprojlim_NH^{m}(U, \bbLog^N(g)) \cong 
		\begin{cases}
			\bbR(0)  & m=g,\\
			\prod_{k=0}^\infty \Sym^k \bbR(\bone)  &   m=2g-1,\\
			0  & \text{otherwise}.
		\end{cases}
	\]
	For $g=1$, we have canonical isomorphisms
	\[
		H^{m}(U, \bbLog(1)) \cong 
		\begin{cases}
			\bbR(0)\oplus\prod_{k=0}^\infty \bbR(k)  &   m=1,\\
			0  & \text{otherwise}.
		\end{cases}
	\]
\end{proposition}

Proposition \ref{prop: Log on U} follows immediately from
the following Lemma \ref{lemma: Log on U}
concerning the cohomology of $U$ with coefficients in $\bbLog^N$
and Lemma \ref{lemma: Log on X}.

\begin{lemma}\label{lemma: Log on U}
	For $g>1$, we have canonical isomorphisms
	\[
		H^{m}(U, \bbLog^N(g)) \cong 
		\begin{cases}
			H^m(X,\bbLog^N(g))  &   m=0,\ldots,g,\\
			\prod_{k=0}^N \Sym^k\bbR(\bone)& m=2g-1,\\
			0  & \text{otherwise}.
		\end{cases}
	\]
	For $g=1$, we have canonical isomorphisms
	\[
		H^{m}(U, \bbLog^N(1)) \cong 
		\begin{cases}
			\bbR(0)\oplus\prod_{k=0}^N \bbR(k)  &   m=1,\\
			0  & \text{otherwise}.
		\end{cases}
	\]
\end{lemma}

\begin{proof}
	If we denote by $i: Z \hookrightarrow X$ the natural inclusion.
	then for any integer $N\geq 0$, we have a long exact sequence
	of mixed $\bbR$-Hodge structures
	\begin{equation}\label{eq: Gyzin}
		\cdots\rightarrow H^m(X, \bbLog^N(g)) \rightarrow  H^m(U, \bbLog^N(g)) \xrightarrow{\res_1}  H^{m+1-2g}(Z, i^*\bbLog^N) \rightarrow \cdots.
	\end{equation}

	Since $i^*\bbLog^N= \prod_{k=0}^N\Sym^k \bbR(\bone)$, we have
	\[
		H^m(Z, i^*\bbLog^N)\cong 
		\begin{cases}
		\prod_{k=0}^N\Sym^k \bbR(\bone)	&  m=0,\\
		0  & m\neq 0.
		\end{cases}
	\]
	Since the $m$-th cohomology of $\bbLog^N(g)$ on $X$ is zero for $m>g$,
	the long exact sequence \eqref{eq: Gyzin} gives the desired result.
	For the case $g=1$, we have the short exact sequence
	\[
		0 \rightarrow H^1(X, \bbLog^N(1)) \rightarrow H^1(U,\bbLog^N(1)) \rightarrow H^0(Z, i^* \bbLog^N) \rightarrow 0.
	\]
	By the diagram
	\[\xymatrix{
		H^1(X, \bbLog^N(1)) \ar[r]\ar[d]^\cong& H^1(U,\bbLog^N(1))\ar[d]\\
		H^1(X, \bbLog^0(1)) \ar[r]\ar[rd]^{\res_0}_\cong& H^1(U,\bbLog^0(1))\ar[d]^{\res_0}\\
		&  \bbR(0),
	}
	\]
	where the left vertical isomorphism follows from Lemma \ref{lemma: Log on X},
	we see that there exists a natural splitting $H^1(U,\bbLog^N(1)) \rightarrow H^1(X,\bbLog^N(1))$ of the first injection.
	Our assertion follows from Proposition \ref{prop: Log on X}.
\end{proof}

\begin{definition}\label{def: geometric}
	We define the \emph{geometric polylogarithm class} to be the class
	\[
		\bsxi = (\bsxi_N)\in H^{2g-1}(U,  \bbLog(g))
	\]
	which maps to the basis $\ul u^{\boldsymbol 0}$ of $\Sym^0 \bbR(\bone)$ through the isomorphism
	 of Proposition \ref{prop: Log on U}.
\end{definition}

\begin{remark}\label{rem: geometric}
	The geometric polylogarithm class $\bsxi=(\bsxi_N)$ satisfies
	\[
		\bsxi_N \in W_0H^{2g-1}(U,\bbLog^N(g))\cap (F^0\cap W_0) H^{2g-1}(U,\bbLog^N_\bbC(g)),
	\]
	since the residue map is a morphism of mixed $\bbR$-Hodge structures.
	Hence we may view this class as an element in 
	\[
		\Hom_{\MHS_\bbR}(\bbR(0), H^{2g-1}(U,\bbLog(g)))\coloneqq
		\varprojlim_N\Bigl(\Hom_{\MHS_\bbR}(\bbR(0), H^{2g-1}(U,\bbLog^N(g)))\Bigr).
	\]
\end{remark}

%
%
\subsection{Explicit Construction of the Geometric Polylogarithm}
%
%

In this subsection, we give an explicit construction of the geometric polylogarithm class.
We first define certain complexes to explicitly calculate the mixed $\bbR$-Hodge structure on the 
cohomology of $U:= X \setminus Z$ for $Z = \{1\} \subset X$, 
and we give an explicit description of the residue map $\res_1$ of \eqref{eq: Gyzin}. 
We let $[g]:=\{1,\ldots,g\}$, and for any $\mu \in [g]$, we let $Z_\mu$ be the divisor 
\[
	Z_\mu := \bbP^1 \times \cdots \times\bbP^1\times\{1\}\times\bbP^1\times\cdots\times\bbP^1 \subset \ol X
\]
with $\{1\}$ in the $\mu$-th component and $\bbP^1$ in the other components.  For any $\bsmu \subset [g]$,
we let $D^\bsmu := D \cup \bigcup_{\mu\in\bsmu} Z_\mu$. 
In what follows, let $N$ be an integer $\geq 0$.

\begin{definition}
	We define the \v Cech-de Rham complex $\cLog^N(D)\otimes\Omega^\bullet_{\ol X}(D^\bullet)$ on $\ol X$ to be the simple complex associated to the double complex
	\[
		\bigoplus_{\abs{\bsmu}=1} \cLog^N(D) \otimes \Omega^\bullet_{\ol X}(D^\bsmu) \xrightarrow{\partial}\cdots\xrightarrow{\partial}\bigoplus_{\abs{\bsmu}=g} \cLog^N(D) \otimes \Omega^\bullet_{\ol X}(D^\bsmu),
	\]
	with $\bigoplus_{\abs{\bsmu}=1} \cLog^N(D) \otimes \Omega^0_{\ol X}(D^\bsmu)$ in degree $0$,
	where the horizontal differentials are the standard alternating sum $(\partial u)_{\mu_1\cdots\mu_h} = \sum_{\nu=1}^h (-1)^\nu u_{\mu_1\cdots\check\mu_\nu\cdots\mu_h}$.
\end{definition}

\begin{proposition}\label{prop: open complex}
	Denote by $j: U \hookrightarrow \ol X$ be the natural inclusion, and consider the $\bbC$-local system $\bbLog^N_\bbC:=\bbLog^N\otimes\bbC$ on $U$.
	Then we have canonical isomorphisms
	\[
		 R j_* \bbLog^N_\bbC \cong R j_*(\cLog^N \otimes \Omega^\bullet_U)\cong\cLog^N(D)\otimes\Omega^\bullet_{\ol X}(D^{\bullet})
	\]
	in the derived category of abelian sheaves on $\ol X$.  In particular, by taking the functor $R\Gamma(\ol X,-)$, we have
	canonical isomorphisms
	\begin{equation}\label{eq: isomorphisms}
	 	H^m(U, \bbLog^N)\otimes\bbC = H^m(U, \bbLog^N_\bbC) \cong H^m_\dR(U, \cLog^N)\cong\bbH^m(\ol X,  \cLog^N(D)\otimes\Omega^\bullet_{\ol X}(D^{\bullet})).
	\end{equation}
\end{proposition}

\begin{proof}
	Since $\cLog^N = \bbLog^N_\bbC \otimes \cO_X$, the complex $\cLog^N\otimes\Omega^\bullet_U$ on $U$ gives a resolution
	\begin{equation}\label{eq: resolution}
		0 \rightarrow \bbLog^N_\bbC \rightarrow \cLog^N \otimes \cO_U \xrightarrow\nabla \cLog^N \otimes \Omega^1_U \rightarrow \cdots
	\end{equation}
	of $\bbLog^N_\bbC$, which proves the first isomorphism of our assertion.
	For any $\mu\in[g]$, we let $U_\mu: = X \setminus Z_\mu = \ol X \setminus (D\cup Z_\mu)$.
	Then $\{U_\mu\}$ is an open covering of $U$.  For any $\bsmu\subset[g]$, we let 
	$U_\bsmu:= \bigcap_{\mu\in\bsmu} U_{\mu}  = \ol X \setminus D^\bsmu$
	and we denote by $j^\circ_\bsmu: U_\bsmu \hookrightarrow U$ and $j_\bsmu := j \circ j^\circ_\bsmu: U_\bsmu\hookrightarrow\ol X$  the natural inclusions.
	If we let $j^\circ_{\bullet*} (\cLog^N \otimes \Omega^\bullet_{U_\bullet})$  be the \v Cech complex  defined as the simple complex associated to the double complex
	\[
		 \prod_{|\bsmu|=1} j^\circ_{\bsmu*} (\cLog^N \otimes \Omega^\bullet_{U_\bsmu})
		\rightarrow \cdots\rightarrow \prod_{|\bsmu|=g} j^\circ_{\bsmu*} (\cLog^N \otimes \Omega^\bullet_{U_\bsmu}),
	\]
	then we have a quasi-isomorphism
	\[
		\cLog^N \otimes \Omega^\bullet_{U} \xrightarrow\cong j^\circ_{\bullet*} (\cLog^N \otimes \Omega^\bullet_{U_\bullet}).
	\]
	This quasi-isomorphism combined with \eqref{eq: resolution} gives the isomorphism
	\begin{equation}\label{eq: first resolution}
		Rj_* \bbLog^N_\bbC \xrightarrow\cong Rj_* j^\circ_{\bullet*} (\cLog^N \otimes \Omega^\bullet_{U_\bullet})
		\cong j_{\bullet*} (\cLog^N \otimes \Omega^\bullet_{U_\bullet})
	\end{equation}
	in the derived category, where  we let 
	$j_{\bullet*} (\cLog^N \otimes \Omega^\bullet_{U_\bullet}) :=  j_* j^\circ_{\bullet*} (\cLog^N \otimes \Omega^\bullet_{U_\bullet})$.
	The second isomorphism of \eqref{eq: first resolution} follows from the fact that $j_\bsmu$ is an affine morphism for any $\bsmu\subset [g]$.
	Note that since $\cLog^N(D)$ is the canonical extension of $\cLog^N$, we have quasi-isomorphisms
	\[
		\cLog^N(D)\otimes\Omega^\bullet_{\ol X}(D^\bsmu)	
		\xrightarrow\cong
		j_{\bsmu*} (\cLog^N \otimes \Omega^\bullet_{U_\bsmu}) 
	\]
	for any $\bsmu\subset [g]$, which induce the quasi-isomorphism 
	\begin{equation}\label{eq: second resolution}
		\cLog^N(D)\otimes\Omega^\bullet_{\ol X}(D^\bullet) \xrightarrow\cong j_{\bullet*} (\cLog^N \otimes \Omega^\bullet_{U_\bullet}).
	\end{equation}
	Our assertion follows by combining the quasi-isomorphisms \eqref{eq: first resolution} and \eqref{eq: second resolution}.
\end{proof}

We next define the filtrations $W_\bullet$ and $F^\bullet$ on $\cLog^N(D)\otimes\Omega^\bullet_{\ol X}(D^\bullet)$ which gives the weight and Hodge filtrations of 
the mixed $\bbR$-Hodge structure $H^m(U, \bbLog^N)$.
For any integer $n\in\bbZ$, we let $W_n(\cLog^N(D)\otimes\Omega^\bullet_{\ol X}(D^{\bullet}))$ be the subcomplex of $\cLog^N(D)\otimes\Omega^\bullet_{\ol X}(D^{\bullet})$
defined as the simple complex associated to the double complex
\[
	\bigoplus_{|\bsmu|=1} W_n(\cLog^N(D)\otimes\Omega^\bullet_{\ol X}(D^{\bsmu})) \rightarrow 
	 \cdots \rightarrow  \bigoplus_{|\bsmu|=g} W_{n+g-1}(\cLog^N(D)\otimes\Omega^\bullet_{\ol X}(D^{\bsmu})),
\]
and for any integer $p\in\bbZ$, we let $F^p(\cLog^N(D)\otimes\Omega^\bullet_{\ol X}(D^{\bullet}))$ be the subcomplex of $\cLog^N(D)\otimes\Omega^\bullet_{\ol X}(D^{\bullet})$
defined as the simple complex associated to the double complex
\[
	\bigoplus_{|\bsmu|=1} F^p(\cLog^N(D)\otimes\Omega^\bullet_{\ol X}(D^{\bsmu})) \rightarrow \cdots \rightarrow 
	 \bigoplus_{|\bsmu|=g} F^p(\cLog^N(D)\otimes\Omega^\bullet_{\ol X}(D^{\bsmu})).
\]
Then the filtrations $\wt W_\bullet$ and $F^\bullet$ induced on $\bbH^m(\ol X, \cLog^N(D)\otimes\Omega^\bullet_{\ol X}(D^\bullet))$ from the filtrations $W_\bullet$ and $F^\bullet$ 
on  $\cLog^N(D)\otimes\Omega^\bullet_{\ol X}(D^\bullet)$ coincide through the isomorphism \eqref{eq: isomorphisms} with the weight and Hodge filtrations 
on the mixed $\bbR$-Hodge structure $H^m(U, \bbLog^N)$ of Proposition \ref{prop: MHS}.

We now describe the residue map $\res_1$ of \eqref{eq: Gyzin}.
Consider $\bbLog^N(g)= \bbLog^N \otimes \bbR(g)$ and $\cLog^N(D)(g)= \cLog^N(D) \otimes \bbC(g)$, and let $\ul\omega^\bsk(g):= \ul\omega^\bsk\otimes\ul\omega^g$.
Note that $Z = Z_1 \cap\cdots\cap Z_g$ and we let $i: Z \hookrightarrow \ol X$ be the natural inclusion.  Consider the map
\[
	\cLog^N(D)(g) \otimes \Omega^g_{\ol X}(D^{[g]}) \rightarrow i_* \bigl( i^* \cLog^N(D) \otimes  \Omega^0_Z\bigr)
\]
given locally in a neighborhood of $Z$ by
\[
	f \otimes \ul\omega^g \otimes \frac{dt_1}{t_1-1} \wedge \cdots \wedge \frac{dt_g}{t_g-1} \mapsto f|_Z \otimes 1,
\]
where $t_1,\ldots,t_g$ is the standard parameter on $\ol X=\bbP^1\times\cdots\times\bbP^1$
so that the divisor $Z_\mu$ is defined by $t_\mu=1$.
This map induces a map of complexes
\[
	 \cLog^N(D)(g) \otimes \Omega^\bullet_{\ol X}(D^\bullet) \rightarrow  i_* \bigl( i^* \cLog^N(D)  \otimes \Omega^\bullet_Z\bigr)[-2g+1],
\]
which induces a map on cohomology
\[
	 \bbH^{2g-1}(\ol X,  \cLog^N(D)(g) \otimes \Omega^\bullet_{\ol X}(D^\bullet)) \rightarrow \bbH^0(Z,  i^* \cLog^N(D)  \otimes \Omega^\bullet_Z).
\]
Then through the isomorphism of Proposition \ref{prop: open complex}, we obtain the map
\[
	\res_1: H^{2g-1}(U, \bbLog^N_\bbC(g)) \rightarrow H^0(Z, i^* \bbLog^N_\bbC).
\]
The map coincides with the residue map $\res_1$ of \eqref{eq: Gyzin}.

Note that the inclusion induces a natural morphism of complexes
\[
	\cLog^N(D)(g) \otimes \Omega^g_{\ol X}(D^{[g]})[-2g+1] \rightarrow \cLog^N(D)(g) \otimes \Omega^\bullet_{\ol X}(D^\bullet),
\]
which induces a map
\[
	\Gamma(\ol X, \cLog^N(D)(g) \otimes \Omega^g_{\ol X}(D^{[g]})) \rightarrow  \bbH^{2g-1}(\ol X, \cLog^N(D)(g) \otimes \Omega^\bullet_{\ol X}(D^\bullet)).
\]

\begin{proposition}
	The sections
	\begin{equation}\label{eq: differential form}
		\xi_N\coloneqq\ul \omega^{\boldsymbol 0}(g) \otimes \frac{dt_1}{t_1-1}\wedge\cdots\wedge \frac{dt_g}{t_g-1} \in \Gamma(\ol X, \cLog^N(D)(g) \otimes \Omega^g_{\ol X}(D^{[g]}))
	\end{equation}
	for integers $N\geq 0$
	define classes 
	\[	[\xi_N] \in
	\bbH^{2g-1}(\ol X, \cLog^N(D)(g) \otimes \Omega^\bullet_{\ol X}(D^\bullet))\cong H^{2g-1}(U, \bbLog^N_\bbC)
	\]
	which are mapped by $\res_1$ to the elements
	$\ul\omega^{\boldsymbol 0} = \ul u^{\boldsymbol 0}$ of $H^0(Z, i^* \bbLog^N) = \prod_{k=0}^N \Sym^k \bbR(\bone)$,
	and are mapped to zero by $\res_0$ in the case $g=1$.
	In other words, $\bsxi\coloneqq([\xi_N])_N$ is the geometric polylogarithm class.
\end{proposition}

The fact that the residue maps preserve the real structure of mixed $\bbR$-Hodge structures
implies that
$[\xi_N]$ is a class in $H^{2g-1}(U, \bbLog^N(g)) \subset H^{2g-1}(U, \bbLog^N_\bbC(g))$.

%
%
%
%
\section{Logarithmic Dolbeault Complex}\label{section: log Dolbeault}
%
%
%
%

In this section, we will use the theory of logarithmic Dolbeault complex defined by Burgos to construct an explicit complex
 calculating the mixed $\bbR$-Hodge structure on the cohomology of $U$
with coefficients in $\bbLog$.
For a general complex manifold $X$, we denote by $\sA_{X,\bbR}$ (resp.  $\sE_{X,\bbR}$) the sheaf of real-valued real analytic (resp. $C^\infty$) functions on $X$, 
and by $\sA^\bullet_{X,\bbR}$ (resp. $\sE^\bullet_{X,\bbR}$) the corresponding complex of sheaves of differential forms.

%
%
\subsection{Review of Logarithmic Dolbeault Complex}
%
%

We first review the theory of logarithmic Dolbeault complex defined by Burgos.
In what follows, let $X$ be a smooth algebraic variety over $\bbC$.   
The $C^\infty$-de Rham complex gives a resolution
\begin{equation}\label{eq: infinity resolution}
	0  \rightarrow \bbR \rightarrow \sE_{X,\bbR} \rightarrow  \sE^1_{X,\bbR} \rightarrow \sE^2_{X,\bbR}  \rightarrow \cdots
\end{equation}
of the constant sheaf $\bbR$ on $X$ in terms of the fine sheaves $\sE^\bullet_{X,\bbR}$.  Burgos defined a logarithmic version of $\sE^\bullet_{X,\bbR}$ on $\ol X$
amenable for defining the weight and Hodge filtrations.
We first define the complex $\sA^\bullet_{\ol X,\bbR}(D)$ to be the $\sA_{\ol X,\bbR}$-subalgebra of $j_*\sA^\bullet_{X,\bbR}$ generated locally on coordinate neighborhoods
adapted to $D$ by  $\sA^\bullet_{\ol X,\bbR}$ and the sections
\begin{align}\label{eq: weight one}
	&\log|t_\mu|, &  &\Re \frac{dt_\mu}{t_\mu},  & & \Im\frac{dt_\mu}{t_\mu}
\end{align}
for $\mu\in \bsmu$ and $\Re\,dt_\mu$, $\Im\,dt_\mu$ for $\mu\not\in \bsmu$, where the set $\bsmu := \{\mu_1,\ldots, \mu_h\} \subset [g]$ is such that
$D$ is defined locally by $t_{\mu_1} \cdots t_{\mu_h} = 0$.
We define the weight filtration $W_\bullet$ on $\sA^\bullet_{\ol X,\bbR}$
to be the multiplicative ascending filtration obtained by assigning weight $0$ to the sections of $\sA_{\ol X,\bbR}$ and weight $1$
to the local sections given in \eqref{eq: weight one}.

Following Burgos \cite{Bur94}*{\S 2}, we define the complex $\sE^\bullet_{\ol X,\bbR}(D)$ to be the image of the natural map
\[
	 \sA^\bullet_{\ol X,\bbR}(D) \otimes_{\sA_{\ol X,\bbR}} \sE_{\ol X,\bbR} \rightarrow j_*  \sE^\bullet_{X,\bbR}.
\]
We let $\sE_{\ol X,\bbR}(D):= \sE^0_{\ol X,\bbR}(D)$.
The complex  $\sE^\bullet_{\ol X,\bbR}(D)$ has a weight filtration $W_\bullet$ induced from the weight filtration $W_\bullet$ on $\sA^\bullet_{\ol X}(D)$.
The complex $\sE^\bullet_{X}:=\sE^\bullet_{X,\bbR}\otimes\bbC$ has a bigrading 
\[
	\sE^m_{X}=\bigoplus_{p,q\in\bbZ, p+q=m} \sE^{p,q}_{X}
\]
induced from the complex structure of $X$, and this bigrading induces a bigrading 
\begin{align*}
	\sE^m_{\ol X}(D) &= \bigoplus_{p,q\in\bbZ, p+q=m} \sE^{p,q}_{\ol X}(D), &
	\sE^{p,q}_{\ol X}(D) &:= \sE^{p+q}_{\ol X}(D) \cap j_* \sE^{p,q}_{X}
\end{align*}
on $\sE^\bullet_{\ol X}(D):=\sE^\bullet_{\ol X,\bbR}(D)\otimes\bbC$.  We define the Hodge filtration on the complex $\sE^\bullet_{\ol X}(D)$ by 
\[
	F^p \sE^\bullet_{\ol X}(D) = \bigoplus_{r\geq p} \sE^{r,s}_{\ol X}(D).
\]
Then we have the following.

\begin{theorem}[\cite{Bur94}*{Theorem 1.2, Theorem 2.1}]\label{theorem: Bu}
 	There exist filtered quasi-isomorphisms
	\[
		(Rj_*\bbR,\tau_{\leq\bullet}) \xrightarrow\cong (j_* \sE^\bullet_{X,\bbR}, \tau_{\leq\bullet}) 
		\xleftarrow\cong
		 (\sE^\bullet_{\ol X,\bbR}(D), \tau_{\leq\bullet}) \xrightarrow\cong
		 (\sE^\bullet_{\ol X,\bbR}(D),W_\bullet)
	\]
	which induces an isomorphism $(Rj_*\bbR,\tau_{\leq\bullet}) \xrightarrow\cong (\sE^\bullet_{\ol X,\bbR}(D),W_\bullet)$ in the filtered derived category,
	and a bifiltered quasi-isomorphism
	\begin{equation}\label{eq: resolution2}
		(\Omega^\bullet_{\ol X}(D),W_\bullet,F^\bullet) \xrightarrow\cong ( \sE^\bullet_{\ol X}(D),W_\bullet,F^\bullet).
	\end{equation}
\end{theorem}

The proof of \eqref{eq: resolution2} is based on the following result, noting that $\sE^\bullet_{\ol X}(D)$ is the simple complex
associated to the double complex $\sE^{\bullet,\bullet}_{\ol X}(D)$.

\begin{proposition}[\cite{Bur94} Proposition 2.3]\label{prop: Burgos 2.3}
	For any $n,p\in \bbZ$, we have the exact sequence
	\[\xymatrix{%
		0\ar[r]& \Gr^W_n \Omega^p_{\ol X}(D) \ar[r]& \Gr^W_n \sE^{p,0}_{\ol X}(D) \ar[r]^{\ol\partial} & \Gr^W_n \sE^{p,1}_{\ol X}(D) \ar[r]^{\quad\ol\partial} & \cdots.
	}\]
\end{proposition}

The advantage of using this complex is that the sheaf $\sE_{\ol X,\bbR}$ is fine, which gives the following.

\begin{lemma}\label{lem: acyclic}
	For integers $m,n,p\in\bbZ$, the sheaves $\Gr^W_n \sE^m_{\ol X,\bbR}(D)$ and $\Gr^W_n \Gr^p_F \sE^m_{\ol X}(D)$ are 
	acyclic with respect to the global section functor $\Gamma(\ol X, -)$.
\end{lemma}

\begin{proof}
	Since the sheaf $\sE_{\ol X,\bbR}$ has a partition of unity, the $\sE_{\ol X,\bbR}$-modules
	$\Gr^W_n \sE^m_{\ol X,\bbR}(D)$ and $\Gr^W_n \Gr^p_F \sE^m_{\ol X}(D)$ for $m, n, p\in\bbZ$ are fine,
	hence is acyclic with respect to the functor $\Gamma(\ol X,-)$.
\end{proof}

%
%
\subsection{Cohomology of the Logarithm Sheaf}
%
%

In this subsection, we first define the $\sE_{\ol X,\bbR}$-module with connection $\sLog^N(D)$.
We will apply Theorem \ref{theorem: Bu} to this module to calculate the cohomology of $U$ with coefficients in $\bbLog^N$.

The module $\cLog^N(D) \otimes\sE_{\ol X}(D)$ has a connection $\nabla: \cLog^N(D) \otimes\sE_{\ol X}(D) \rightarrow \cLog^N(D) \otimes\sE^1_{\ol X}(D)$
induced from the connection on $\cLog^N(D)$.
If we let 
\begin{equation}\label{eq: e}
	\ul e^\bsk:= i^{\abs{\bsk}} \exp\biggl(\sum_{\nu=1}^g (-\log|t_\nu|)\omega_\nu\biggr)\cdot\ul\omega^\bsk
\end{equation}
for $\bsk\in\bbN^g$, where the operator $\omega_\mu$ acts on $\ul\omega^\bsk$ by $\omega_\mu\cdot\ul\omega^\bsk=\ul\omega^{\bsk + 1_\mu}$,
then we have
\begin{align*}
	\nabla(\ul e^\bsk)& = i^{\abs\bsk} \exp\biggl(\sum_{\nu=1}^g (-\log|t_\nu|)\omega_\nu\biggr)\cdot \sum_{\mu=1}^g \ul\omega^{\bsk+1_\mu} \otimes \frac{dt_\mu}{t_\mu} \\
	&\qquad-i^{\abs{\bsk}} \sum_{\mu=1}^g \exp\biggl( \sum_{\nu=1}^g(-\log|t_\nu|)\omega_\nu\biggr)\omega_\mu\cdot \ul\omega^{\bsk} \otimes
	\frac{1}{2}\biggl( \frac{dt_\mu}{t_\mu} + \frac{d \ol t_\mu}{\ol t_\mu}\biggr)\\
	&= \sum_{\mu=1}^g i^{\abs{\bsk}} \exp\biggl(  \sum_{\nu=1}^g (-\log|t_\nu|)\omega_\nu\biggr)\cdot\ul\omega^{\bsk+1_\nu} \otimes i\Im\frac{dt_\nu}{t_\nu} 
	= \sum_{\mu=1}^g \ul e^{\bsk+1_\mu}\otimes\Im\frac{dt_\mu}{t_\mu},
\end{align*}
that is
\begin{equation}\label{eq: connection e}
	\nabla(\ul e^\bsk) = \sum_{\mu=1}^g \ul e^{\bsk+1_\mu} \otimes \Im \frac{dt_\mu}{t_\mu}.
\end{equation}
This calculation shows that we may define an $\bbR$-structure $\sLog^N(D)$ on $\cLog^N(D)$ as follows.

\begin{definition}\label{def: sLog}
	We let $\sLog^N(D)$ be the free $\sE_{\ol X,\bbR}$-module generated by 
	$
		\ul e^\bsk
	$
	of \eqref{eq: e}
	for $\bsk\in\bbN^g$ such that $\abs{\bsk}\leq N$, with connection
	\[
		\nabla: \sLog^N(D) \rightarrow \sLog^N(D) \otimes\sE^1_{\ol X,\bbR}(D)
	\]
	given by \eqref{eq: connection e}.
	By definition,  $\sLog^N(D)$ is an $\sE_{\ol X,\bbR}$-submodule of $\cLog^N(D) \otimes\sE_{\ol X}(D)$.
\end{definition}

We define the filtration $W_\bullet$ on $\sLog^N(D)$ such that $\ul e^\bsk$ is of weight $-2\abs\bsk$.
The basis $\ul\omega^\bsk$ of $\cLog^N(D)$ in terms of $\ul e^\bsk$ is given by
\begin{equation}\label{eq: base change two}
	\ul\omega^\bsk = (-i)^{\abs\bsk} \exp\biggl(\sum_{\nu=1}^g\log|t_\nu|\omega_\nu\biggr)\cdot\ul e^\bsk,
\end{equation} 
where  $\omega_\mu$ acts on $\ul e^\bsk$ as
\begin{align*}
	\omega_\mu \cdot \ul e^\bsk &= i^{\abs{\bsk}} \exp\biggl( \sum_{\nu=1}^g(-\log|t_\nu|)\omega_\nu\biggr)\omega_\mu\cdot\ul\omega^\bsk 
	=  i^{\abs{\bsk}} \exp\biggl( \sum_{\nu=1}^g(-\log|t_\nu|)\omega_\nu\biggr)\cdot\ul\omega^{\bsk+1_\mu}\\
	&= (-i) i^{\abs\bsk+1} \exp\biggl( \sum_{\nu=1}^g(-\log|t_\nu|)\omega_\nu\biggr)\cdot\ul\omega^{\bsk+1_\mu} = (-i) \ul e^{\bsk+1_\mu}.
\end{align*}
Note that we have an exact sequence
\begin{equation}\label{eq: SES Log^N Bu}
	0 \rightarrow \Sym^{N} \bbR(\bone) \otimes \sE_{\ol X,\bbR} \rightarrow \sLog^N(D) \rightarrow \sLog^{N-1}(D)\rightarrow 0.
\end{equation}
compatible with the connection and $W_\bullet$.
The natural inclusion induces an isomorphism of projective systems
\begin{equation}\label{eq: iota}
	 \sLog^N(D) \otimes \sE_{\ol X}(D)  = \cLog^N(D)\otimes \sE_{\ol X}(D),
\end{equation}
which is compatible with the connection and the filtration $W_\bullet$.
The relation between the basis $(\ul u^\bsk)_{\bsk\in\bbN^g}$ of $\bbLog^N$ and $(\ul e^\bsk)_{\bsk\in\bbN^g}$ of $\sLog^N(D)$ is given by
\begin{align*}
	\ul u^\bsk &= (2\pi i)^{\abs\bsk} \exp\biggl( \sum_{\nu=1}^g (- \log t_\nu) \omega_\nu\biggr)\cdot\ul\omega^\bsk = (2\pi)^{\abs\bsk}  
	\exp\biggl(  \sum_{\nu=1}^g (- i \Im(\log t_\nu)) \omega_\nu\biggr)\cdot\ul e^\bsk\\
	 &= (2\pi)^{\abs\bsk}  \exp\biggl(\sum_{\nu=1}^g (-\Im(\log t_\nu))e_\nu\biggr)\cdot\ul e^\bsk,
\end{align*}
where $e_\nu$ acts on $\ul e^\bsk$ by $e_\nu\cdot\ul e^{\bsk} = \ul e^{\bsk+1_\nu}$, the equality given locally for each choice of a branch of $\log t_\nu$.
Since both $(2\pi)^{\abs\bsk}$ and $\Im(\log t_\nu)$ are real-valued functions on $X$, this gives an isomorphism
\begin{equation}\label{eq: BE two}
	\bbLog^N \otimes \sE_{X,\bbR} \cong \sLog^N := \sLog^N(D)|_X
\end{equation}
which is compatible with the connection and the filtration $W_\bullet$.
Moreover, since $\Gr^W_n \bbLog^N$ is a constant variation of pure Hodge structure on $X$, this naturally
extends to a constant variation of pure Hodge structures on $\ol X$, and the isomorphism induced on $\Gr^W_n$ by \eqref{eq: BE two} extends to an isomorphism
\begin{equation}\label{eq: BE three}
	\Gr^{W}_n \sLog^N(D) \cong \Gr^W_n \bbLog^N  \otimes \sE_{\ol X,\bbR}
\end{equation}
on $\ol X$.

\begin{definition}
	We define the logarithmic \v{C}ech-Dolbeault complex $\sLog^N(D)\otimes\sE^\bullet_{\ol X,\bbR}(D^\bullet)$ on $\ol X$ to be the simple complex associated to the double complex
	\[
		\bigoplus_{\abs{\bsmu}=1} \sLog^N(D) \otimes \sE^\bullet_{\ol X,\bbR}(D^\bsmu) \xrightarrow{\partial}\cdots
		\xrightarrow{\partial}\bigoplus_{\abs{\bsmu}=g} \sLog^N(D) \otimes \sE^\bullet_{\ol X,\bbR}(D^\bsmu)
	\]
	with $\bigoplus_{\abs{\bsmu}=1} \sLog^N(D) \otimes \sE_{\ol X,\bbR}(D^\bsmu)$ in degree $0$,
	where the horizontal differentials are the standard alternating sums.
\end{definition}

\begin{proposition}\label{prop: real structure}
	Let $j: U \hookrightarrow \ol X$ be the natural inclusion.  Then we have canonical isomorphisms
	\begin{equation}\label{eq: real isomorphism}
		Rj_* \bbLog^N \cong Rj_* (\bbLog^N \otimes \sE^\bullet_{U,\bbR})\cong\sLog^N(D) \otimes \sE^\bullet_{\ol X,\bbR}(D^\bullet)
	\end{equation}
	in the derived category of abelian sheaves on $\ol X$.   In particular, by taking the functor $R\Gamma(\ol X, -)$, we obtain
	a canonical isomorphism
	\[
	 	H^m(U, \bbLog^N) \cong  H^m(\Gamma(\ol X,\sLog^N(D)\otimes\sE^\bullet_{\ol X,\bbR}(D^\bullet))).
	\]
\end{proposition} 

\begin{proof}
	Since $\bbLog^N$ is an $\bbR$-local system on $U$, the resolution \eqref{eq: infinity resolution} gives a resolution
	\[
		0 \rightarrow \bbLog^N \rightarrow \bbLog^N \otimes \sE_{U,\bbR} \rightarrow \bbLog^N\otimes\sE^1_{U,\bbR} \rightarrow \cdots
	\]
	of $\bbLog^N$ on $U$.  By taking $Rj_*$, we obtain the first isomorphism of \eqref{eq: real isomorphism}.  
	As in the proof of Proposition \ref{prop: open complex},
	for $\mu\in[g]$, we let $U_\mu= X \setminus Z_\mu = \ol X \setminus (D \cup Z_\mu)$, which gives an open covering $\{ U_\mu \}$ of $U$.
	For $\bsmu \subset [g]$, we let $U_\bsmu = \ol X \setminus D^\bsmu = \bigcap_{\mu\in\bsmu} U_\mu$, and we let $j_\bsmu : U_\bsmu\hookrightarrow \ol X$
	be the natural inclusion.  If we let $j_{\bullet*}(\bbLog^N \otimes \sE^\bullet_{U_\bullet,\bbR})$  be the \v Cech complex  defined as the simple complex associated to the double complex
	\[
		\prod_{\abs{\bsmu}=1} j_{\bsmu*} (\bbLog^N \otimes \sE^\bullet_{U_\bsmu,\bbR}) \rightarrow\cdots\rightarrow\prod_{\abs{\bsmu}=g} j_{\bsmu*} (\bbLog^N \otimes \sE^\bullet_{U_\bsmu,\bbR}).
	\]
	Following a similar argument to that in \eqref{eq: first resolution}, noting that $\bbLog^N \otimes \sE^\bullet_{U_\bsmu,\bbR}$ are complexes of fine sheaves,  we have a quasi-isomorphism
	\begin{equation}\label{eq: second}
		R j_{*}(\bbLog^N \otimes \sE^\bullet_{U,\bbR})   \xrightarrow\cong j_{\bullet*}(\bbLog^N \otimes \sE^\bullet_{U_\bullet,\bbR}).
	\end{equation}
	In order to complete our proof, we prove by induction on $N \geq 0$ that for any $\bsmu \subset [g]$, the natural morphism
	\begin{equation}\label{eq: third}
		\sLog^N(D) \otimes \sE^\bullet_{\ol X,\bbR}(D^\bsmu)  \rightarrow  j_{\bsmu*}(\bbLog^N \otimes \sE^\bullet_{U_\bsmu,\bbR})
	\end{equation}
	is a quasi-isomorphism.  If $N=0$, then $\bbLog^0 \cong \bbR$ and $\sLog^0(D) \cong \sE_{\ol X,\bbR}$, hence \eqref{eq: third} is a quasi-isomorphism by 
	Theorem \ref{theorem: Bu}.  Suppose \eqref{eq: third} is a quasi-isomorphism for $N-1 \geq 0$.
	The short exact sequence of \eqref{eq: SES Log^N Bu} gives the commutative diagram
	\[\small\xymatrix{
		0 \ar[r]& \Sym^{N}\bbR(\bone)\otimes \sE^\bullet_{\ol X,\bbR}(D^\bsmu)\ar[r]\ar[d]^\cong& \sLog^N(D) \otimes \sE^\bullet_{\ol X,\bbR}(D^\bsmu)
		 \ar[r]\ar[d]& \sLog^{N-1}(D)   \otimes \sE^\bullet_{\ol X,\bbR}(D^\bsmu)\ar[r]\ar[d]^\cong& 0  \\
		0 \ar[r]& j_{\bsmu*}( \Sym^{N}\bbR(\bone)\otimes \sE^\bullet_{U_\bsmu,\bbR}) 
		 \ar[r]& j_{\bsmu*}(\sLog^N \otimes \sE^\bullet_{U_\bsmu,\bbR})\ar[r]& j_{\bsmu*}(\sLog^{N-1} \otimes \sE^\bullet_{U_\bsmu,\bbR}) \ar[r]& 0. 
	}\]
	The first vertical morphism is a quasi-isomorphism by  Theorem \ref{theorem: Bu}, noting that 
	\[
		j_{\bsmu*}( \Sym^{N}\bbR(\bone) \otimes \sE^\bullet_{U_\bsmu,\bbR}) =   \Sym^{N}\bbR(\bone) \otimes j_{\bsmu*} \sE^\bullet_{U_\bsmu,\bbR},
	\]
	and the third vertical morphism is a quasi-isomorphism by the induction hypothesis.
	Hence the middle vertical morphism is a quasi-isomorphism, which by induction proves that \eqref{eq: third} is a quasi-isomorphism for any $N\geq 0$.
	Combining this result with \eqref{eq: second}, we obtain isomorphisms
	\[
		R j_{*}(\bbLog^N \otimes \sE^\bullet_{U,\bbR})  \cong j_{\bullet*}(\bbLog^N \otimes \sE^\bullet_{U_\bullet,\bbR})
		\cong \sLog^N(D) \otimes \sE^\bullet_{\ol X,\bbR}(D^\bullet)
	\]
	in the derived category, which proves the second isomorphism of \eqref{eq: real isomorphism}.  The statement for cohomology follows from 
	\eqref{eq: real isomorphism} and the fact that $\sLog^N\otimes\sE^\bullet_{\ol X,\bbR}(D)$ is a complex of fine sheaves.
\end{proof}

%
%
\subsection{Mixed Hodge Structure of the Logarithm Sheaf}\label{subsection: MHS log}
%
%

In this subsection, we show that the weight and Hodge filtrations on $\sLog^N(D)$ induce the weight and Hodge filtrations of the mixed $\bbR$-Hodge
structure on the cohomology of $U$ with coefficients in $\bbLog^N$.

We define the filtration $W_\bullet$ on $\sLog^N(D)\otimes\sE^\bullet_{\ol X,\bbR}(D^\bsmu)$ and $F^\bullet$ on $\sLog^N(D)\otimes\sE^\bullet_{\ol X}(D^\bsmu)$ by 
\[
	\begin{split}
		W_{n}(\sLog^N(D)\otimes\sE^\bullet_{\ol X,\bbR}(D^\bsmu)) &:= \sum_{k\in\bbZ} W_{n-k}\sLog^N(D)\otimes W_k \sE^\bullet_{\ol X,\bbR}(D^\bsmu), \\
		F^p(\sLog^N(D)\otimes\sE^\bullet_{\ol X}(D^\bsmu)) &:= \sum_{q\in\bbZ} F^{p-q} \cLog^N(D)\otimes F^q\sE^\bullet_{\ol X}(D^\bsmu),
	\end{split}
\]
where we use the equality $\sLog^N(D)\otimes\sE^\bullet_{\ol X}(D^\bsmu) = \cLog^N(D)\otimes\sE^\bullet_{\ol X}(D^\bsmu)$ induced from \eqref{eq: iota} for the definition of $F^\bullet$.
For any integer $n\in\bbZ$, we let $W_n(\sLog^N(D)\otimes\sE^\bullet_{\ol X,\bbR}(D^{\bullet}))$ be the subcomplex of $\sLog^N(D)\otimes\sE^\bullet_{\ol X,\bbR}(D^{\bullet})$
defined as the simple complex associated to the double complex
\[
	\bigoplus_{|\bsmu|=1} W_n(\sLog^N(D)\otimes\sE^\bullet_{\ol X,\bbR}(D^{\bsmu})) \rightarrow 
	 \cdots \rightarrow  \bigoplus_{|\bsmu|=g} W_{n+g-1}(\sLog^N(D)\otimes\sE^\bullet_{\ol X,\bbR}(D^{\bsmu})),
\]
and for any integer $p\in\bbZ$, we let $F^p(\sLog^N(D)\otimes\sE^\bullet_{\ol X}(D^{\bullet}))$ be the subcomplex of $\sLog^N(D)\otimes\sE^\bullet_{\ol X}(D^{\bullet})$
defined as the simple complex associated to the double complex
\[
	\bigoplus_{|\bsmu|=1} F^p(\sLog^N(D)\otimes\sE^\bullet_{\ol X}(D^{\bsmu})) \rightarrow \cdots \rightarrow 
	 \bigoplus_{|\bsmu|=g} F^p(\sLog^N(D)\otimes\sE^\bullet_{\ol X}(D^{\bsmu})).
\]

\begin{theorem}\label{thm: unipotent Bu}
	The natural inclusion
	\[
		\cLog^N(D)\otimes \Omega^\bullet_{\ol X}(D^\bullet) \rightarrow 
		 \cLog^N(D)\otimes \sE^\bullet_{\ol X}(D^\bullet) 
		  =  \sLog^N(D)\otimes \sE^\bullet_{\ol X}(D^\bullet)
	\]
	induces a bifiltered quasi-isomorphism 
	\[
		(\cLog^N(D)\otimes \Omega^\bullet_{\ol X}(D^\bullet), W_\bullet, F^\bullet) \xrightarrow\cong
		( \sLog^N(D)\otimes \sE^\bullet_{\ol X}(D^\bullet), W_\bullet, F^\bullet).
	\]
\end{theorem}

\begin{proof}
	It is sufficient to prove that we have a bifiltered quasi-isomorphism
	\[
		(\cLog^N(D)\otimes \Omega^\bullet_{\ol X}(D^\bsmu), W_\bullet, F^\bullet) \xrightarrow\cong
		( \sLog^N(D)\otimes \sE^\bullet_{\ol X}(D^\bsmu), W_\bullet, F^\bullet)
	\]
	for any $\bsmu\subset[g]$.
	Since the filtration $W_\bullet$ on $\cLog^N(D)$ and $\sLog^N(D)$ are filtrations by free submodules, we have
	\begin{align*}
		\Gr^W_n (\cLog^N(D)\otimes \Omega^\bullet_{\ol X}(D^\bsmu) ) 
		&= \bigoplus_{k\in\bbZ}
		\Gr^W_{n-k} \cLog^N(D)\otimes \Gr^W_{k} \Omega^\bullet_{\ol X}(D^\bsmu) 
		=\bigoplus_{k\in\bbZ}\Gr^W_{n-k} \bbLog^N_\bbC\otimes \Gr^W_{k} \Omega^\bullet_{\ol X}(D^\bsmu) 
		\\
		\Gr^W_n (\sLog^N(D)\otimes \sE^\bullet_{\ol X}(D^\bsmu) ) &= \bigoplus_{k\in\bbZ}
		\Gr^W_{n-k} \sLog^N(D)\otimes \Gr^W_{k} \sE^\bullet_{\ol X}(D^\bsmu)
		=\bigoplus_{k\in\bbZ}\Gr^W_{n-k} \bbLog^N_\bbC\otimes \Gr^W_{k} \sE^\bullet_{\ol X}(D^\bsmu).
	\end{align*}
	Moreover, since the filtration $F^\bullet$ on $\Gr^W_{n-k} \bbLog^N_\bbC$ is a filtration by $\bbC$-local systems,
	 we see that
	\begin{align*}
		\Gr^p_F \Gr^W_n (\cLog^N(D)\otimes \Omega^\bullet_{\ol X}(D^\bsmu) ) 
		&= \bigoplus_{k,q\in\bbZ} \Gr^{p-q}_F \Gr^W_{n-k}\bbLog^N_\bbC \otimes \Gr^W_{k} \Omega^q_{\ol X}(D^\bsmu)
		\\
		\Gr^p_F \Gr^W_n (\sLog^N(D)\otimes \sE^{\bullet}_{\ol X}(D^\bsmu) ) 
		&= \bigoplus_{k,q\in\bbZ}\Gr^{p-q}_F \Gr^W_{n-k}\bbLog^N_\bbC \otimes \Gr^W_{k} \sE^{q,\bullet}_{\ol X}(D^\bsmu).
	\end{align*}
	 Since $\Gr^{p-q}_F \Gr^W_{n-k}\bbLog^N_\bbC$ are constant $\bbC$-local systems on $\ol X$, our assertion follows from Proposition \ref{prop: Burgos 2.3}.
\end{proof}

By Theorem \ref{thm: unipotent Bu}, we see that the complex $\sLog^N(D)\otimes \sE^\bullet_{\ol X,\bbR}(D^\bullet)$ may be used to calculate the mixed $\bbR$-Hodge structure $H^m(U, \bbLog^N)$.
This implies in particular that
\[
	\bigl((\sLog^N(D)\otimes \sE^\bullet_{\ol X,\bbR}(D^\bullet), W_\bullet),  (\sLog^N(D)\otimes \sE^\bullet_{\ol X}(D^\bullet), W_\bullet, F^\bullet), \id\bigr)
\]
is a cohomological mixed $\bbR$-Hodge complex on $\ol X$ in the sense of \cite{Del74}*{(8.1.6)}.  
Furthermore,  by Lemma \ref{lem: acyclic}, $\sLog^N(D)\otimes \sE^\bullet_{\ol X,\bbR}(D^\bullet)$ as well as $\Gr^p_F \Gr^W_n$ 
are complexes of sheaves acyclic with respect to the global section functor, we may simply take the global section to calculate the corresponding
cohomology groups and their filtrations.

\begin{definition}\label{def: complex}
	We let
	\begin{align*}
		R^\bullet(U, \bbLog^N)&:= \Gamma(\ol X, \sLog^N(D)\otimes \sE^\bullet_{\ol X,\bbR}(D^\bullet)), &
		R^\bullet(U, \bbLog^N_\bbC)&:=R^\bullet(U, \bbLog^N)\otimes\bbC
	\end{align*}
	with filtrations $\wt W_\bullet$ and $F^\bullet$ defined by
	\begin{align*}
		\wt W_n R^\bullet(U, \bbLog^N)&:= \Gamma(\ol X, W_n(\sLog^N(D)\otimes \sE^\bullet_{\ol X,\bbR}(D^\bullet))),  \\
		F^p R^\bullet(U, \bbLog^N_\bbC)&:= \Gamma(\ol X, F^p(\sLog^N(D)\otimes \sE^\bullet_{\ol X}(D^\bullet))).
	\end{align*}
	Furthermore, we let $W_\bullet:=  \Dec(\wt W)_\bullet$ be the d\'ecalage of $\wt W_\bullet$ in the sense of \cite{Del71}*{(1.3.3)}, given by
	\[
		W_n R^m(U, \bbLog^N):= \Ker\bigl(d^m:  \wt W_{n-m} R^{m}(U, \bbLog^N) \rightarrow R^{m+1}(U, \bbLog^N)/\wt W_{n-m-1} R^{m+1}(U, \bbLog^N)\bigr).
	\]	
\end{definition}
By \cite{Del74}*{(8.1.9)},  the spectral sequences for the filtrations $W_\bullet$ and $F^\bullet$ degenerates at $E_1$.
This implies in particular that the filtrations $W_\bullet$ and $F^\bullet$ are strictly compatible with the differential maps,
hence
\begin{align}
	W_n H^m( R^\bullet(U, \bbLog^N)) &=H^m( W_n R^\bullet(U, \bbLog^N)), \label{eq: weight2}\\
	\notag	F^q H^m( R^\bullet(U, \bbLog^N)) &=H^m(F^q R^\bullet(U, \bbLog^N)).
\end{align}
By the property of the d\'ecalage, we have
\[
	W_n H^m( R^\bullet(U, \bbLog^N)) = \wt W_{n-m} H^m(R^\bullet(U, \bbLog^N)).
\]

\begin{corollary}\label{cor: 49}
	The complex $(R^\bullet(U, \bbLog^N), W_\bullet, F^\bullet)$ given above calculates the mixed $\bbR$-Hodge structure
	$H^m(U,\bbLog^N)$ given in Proposition \ref{prop: MHS}.
\end{corollary}

\begin{proof}
	By Proposition \ref{prop: real structure}, we have an isomorphism
	\[
		H^m(R^\bullet(U, \bbLog^N)) \cong H^m(U,\bbLog^N),
	\]
	and by Theorem \ref{thm: unipotent Bu}, the filtrations induced from $\wt W_\bullet$ and $F^\bullet$ on 
	$R^\bullet(U, \bbLog^N)$ corresponds through this isomorphism to the filtrations $\wt W_\bullet$ and $F^\bullet$ on $H^m(U,\bbLog^N)$.
	Our assertion follows from $W_\bullet = \wt W_\bullet[m]$. 
\end{proof}

As a variant of $R^\bullet(U, \bbLog^N)$, we let
\begin{align*}
	R^\bullet(U, \bbLog^N(n))&:= \Gamma(\ol X, \sLog^N(D)(n)\otimes \sE^\bullet_{\ol X,\bbR}(D^\bullet)), &
	R^\bullet(U, \bbLog^N_\bbC(n)) &:= R^\bullet(U, \bbLog^N(n)) \otimes\bbC
\end{align*}
for any integer $n\in\bbZ$, where $\sLog^N(D)(n):= \sLog^N(D) \otimes \bbR(n)$ with the corresponding weight and Hodge filtrations.
This filtered complex calculates the mixed $\bbR$-Hodge structure $H^m(U, \bbLog^N(n))$.

The geometric polylogarithm class is given explicitly as follows.
Through the map 
\[
	\cLog^N(D)(g) \otimes \Omega_{\ol X}^g(D^{[g]}) \rightarrow \cLog^N(D)(g) \otimes \sE^{g}_{\ol X}(D^{[g]}),
\]
the differential form
\[
	\xi_N = \ul\omega^{\boldsymbol 0}(g) \otimes \frac{dt_1}{t_1-1} \wedge\cdots\wedge\frac{dt_g}{t_g-1}
\]
of \eqref{eq: differential form} giving the geometric polylogarithm class $[\xi_N]$ in $H^1(U, \bbLog^N(g))$ defines an element 
\[
	\xi_N \in R^{2g-1}(U, \bbLog^N_\bbC(g))= \Gamma(\ol X, \sLog^N(D)(g) \otimes \sE^{g}_{\ol X}(D^{[g]})).
\]
Note that the weight and Hodge filtrations on  $R^{2g-1}(U, \bbLog^N_\bbC(g))$ are given by
\begin{align*}
	W_n R^{2g-1}(U, \bbLog^N_\bbC(g)) &= \wt W_{n-2g+1}R^{2g-1}(U, \bbLog^N_\bbC(g))  = \Gamma(\ol X, W_{n-g}(\sLog^N(D)(g) \otimes \sE^{g}_{\ol X}(D^{[g]}))),  \\
	F^p R^{2g-1}(U, \bbLog^N_\bbC(g)) &=  \Gamma(\ol X, F^p(\sLog^N(D)(g) \otimes \sE^{g}_{\ol X}(D^{[g]}))).
\end{align*}
Since $\xi_N$ is an element in both
\begin{align*}
	 \Gamma(\ol X, W_{-2g} \sLog^N(D)(g) \otimes W_g \sE^{g}_{\ol X}(D^{[g]})) &\subset W_0 R^{2g-1}(U, \bbLog^N_\bbC(g)),\\
	 \Gamma(\ol X, F^{-g} \sLog^N(D)(g) \otimes F^g \sE^{g}_{\ol X}(D^{[g]})) &\subset  F^0R^{2g-1}(U, \bbLog^N_\bbC(g)),
\end{align*}
we see that 
$
	\xi_N \in (F^0 \cap W_0)  R^{2g-1}(U, \bbLog^N_\bbC(g)),
$
hence the class $[\xi_N]$ satisfies
\[
	[\xi_N] \in (F^0 \cap W_0) H^{2g-1}(U, \bbLog^N_\bbC(g)).
\]

%
%
%
%
\section{Absolute Polylogarithm}\label{section: main theorem}
%
%
%
%

The purpose of this section is to define and give an explicit construction of the absolute polylogarithm class
\[
	\pol \in H^{2g-1}_\sA(U, \bbLog(g)).
\]
We start with the definition of the absolute Hodge cohomology $H^{2g-1}_\sA(U, \bbLog(g))$.

%
%
\subsection{The Absolute Polylogarithm Class}
%
%

We let
\[
	R^\bullet(U, \bbLog^N(g))= \Gamma(\ol X, \sLog^N(D)(g)\otimes \sE^\bullet_{\ol X,\bbR}(D^\bullet))
\]
be the logarithmic \v Cech-Dolbeault complex of \S \ref{subsection: MHS log}.   

\begin{definition}\label{def: AHC}	
	We define $R^\bullet_\sA(U, \bbLog^N(g))$ to be the complex
	\begin{align*}
		 \Cone\Bigl(W_0 R^\bullet(U, \bbLog^N(g)) \oplus  (F^0 \cap W_0)R^\bullet(U, \bbLog^N_\bbC(g)) 	\rightarrow
		W_0 R^\bullet(U, \bbLog^N_\bbC(g))  \Bigr)[-1],
	\end{align*}
	where the map is defined by $(x,y) \mapsto x-y$.  Then we define
	the absolute Hodge cohomology $H^m_\sA(U,\bbLog^N(g))$ of $U$ with coefficients in $\bbLog^N(g)$ by
	\begin{equation}\label{eq: AHC}
		H^m_\sA(U,\bbLog^N(g))\coloneqq H^m(R^\bullet_\sA(U, \bbLog^N(g))).
	\end{equation}
	In addition, we let
	\[
		H^m_\sA(U,\bbLog(g))\coloneqq \varprojlim_N H^m_\sA(U,\bbLog^N(g)).
	\]
\end{definition}

\begin{remark}
	By Corollary \ref{cor: 49}, the complex $R^\bullet(U, \bbLog^N(g))$ is a complex of filtered modules which
	gives the mixed $\bbR$-Hodge structure on the cohomology of $U$ with coefficients in $\bbLog^N(g)$.
	In fact, the complex $R^\bullet(U, \bbLog^N(g))$ gives a Hodge complex, and hence 
	an object in the derived category $D^b(\MHS_\bbR)$ via Beilinson's equivalence of categories \cite{Bei86}*{Theorem 3.4}.
	Then the absolute Hodge cohomology \eqref{eq: AHC} may be interpreted as
	\[
		H^m_\sA(U,\bbLog^N(g)) = \Hom_{D^b(\MHS_\bbR)}(\bbR(0),R^\bullet(U, \bbLog^N(g))[m]).
	\]
\end{remark}
By definition of the cone, we have a long exact sequence
\begin{align*}
	\cdots  &\rightarrow W_0 H^{m-1}(U, \bbLog^N(g))  
	\oplus (F^0 \cap W_0)H^{m-1}(U, \bbLog^N_\bbC(g))  \rightarrow W_0 H^{m-1}(U, \bbLog^N_\bbC(g))\\ 
	&\rightarrow H^m_\sA(U,\bbLog^N(g))\\
		 &\rightarrow  W_0 H^m(U, \bbLog^N(g))  \oplus (F^0 \cap W_0)H^m(U, \bbLog^N_\bbC(g)) 
		 \rightarrow W_0 H^m(U, \bbLog^N_\bbC(g)) \rightarrow \cdots.
\end{align*}
Note that for any mixed $\bbR$-Hodge structure $V \coloneqq (V, W_\bullet,F^\bullet)$, the extension groups in $\MHS_\bbR$ are
calculated as 
\begin{equation}\label{eq: Ext}
	\Ext^m_{\MHS_\bbR}(\bbR(0), V) =
	\begin{cases}
		\Ker(W_0 V \oplus (F^0\cap W_0)V_\bbC \rightarrow W_0V_\bbC)  & m=0,\\
		\Coker(W_0 V \oplus (F^0\cap W_0)V_\bbC \rightarrow W_0V_\bbC)  & m=1,\\
		0  &  \text{otherwise},
	\end{cases}
\end{equation}
where the map is defined by $(x,y) \mapsto x-y$.
Thus we obtain the short exact sequence
\begin{multline}\label{eq: SES second}
	0 \rightarrow \Ext^1_{\MHS_\bbR}(\bbR(0), H^{m-1}(U, \bbLog^N(g)))\\  \rightarrow H^m_\sA(U, \bbLog^N(g)) \rightarrow \Hom_{\MHS_\bbR}(\bbR(0), H^{m}(U, \bbLog^N(g)) )
	\rightarrow 0
\end{multline}
for any $m \in \bbZ$.

The following proposition is crucial in defining the polylogarithm class.

\begin{proposition}\label{prop: isomorphism}
	We have a canonical isomorphism
	\[	
		H^{2g-1}_\sA(U, \bbLog(g)) \cong \Hom_{\MHS_\bbR}(\bbR(0), H^{2g-1}(U,\bbLog(g))).
	\]
	Moreover, we have
	\[
		H^{2g-1}_\sA(U, \bbLog(g)) \cong \begin{cases}
		\bbR & g>1 \\ \bbR\oplus\bbR &g=1.\end{cases}
	\]
\end{proposition}

\begin{proof}
	By Lemma \ref{lemma: Log on U}, we have 
	\begin{align*}
		H^{2g-2}(U, \bbLog^N(g))&=
		\begin{cases}
			0  &  g\neq2\\
			\bbR(0) & g=2.
		\end{cases},&
		H^{2g-1}(U, \bbLog^N(g))&=
		\begin{cases}
			\prod_{k=0}^N \Sym^k\bbR(\bone)  &  g\neq 1\\
			\bbR(0)\oplus\prod_{k=0}^N \Sym^k\bbR(\bone)  &  g= 1.
		\end{cases}
	\end{align*} 
	On the other hand, by \eqref{eq: Ext}, we have
	\begin{align*}
		\Hom_{\MHS_\bbR}(\bbR(0), \bbR(n))& \cong
		\begin{cases}
			\bbR  & n=0,\\
			0  &  n \neq 0,
		\end{cases}
		&
		\Ext^1_{\MHS_\bbR}(\bbR(0), \bbR(n))& \cong
		\begin{cases}
		\bbC/(2\pi i)^n \bbR  & n>0,\\
		0 & n \leq 0,
		\end{cases}
	\end{align*}
	and $\Ext^m_{\MHS_\bbR}(\bbR(0),\bbR(n))=0$ if $m\neq 0,1$.
	Our assertion follows from \eqref{eq: SES second} and by passing to the limit.
\end{proof}

The absolute polylogarithm class for $X$ is defined as follows.

\begin{definition}\label{def: polylogarithm}
	We define the \emph{absolute polylogarithm class} $\pol$ to be the class
	\[
		\pol =(\pol_N)\in H^{2g-1}_\sA(U,  \bbLog(g))=\varprojlim_NH^{2g-1}_\sA(U,  \bbLog^N(g))
	\]
	which maps to the geometric polylogarithm class
	in $\Hom_{\MHS_\bbR}(\bbR(0), H^{2g-1}(U,  \bbLog(g)))$
	of Remark \ref{rem: geometric}
	through the isomorphism of Proposition \ref{prop: isomorphism}.
\end{definition}

\begin{remark}
	When $g>1$, then $\pol_N\in H^{2g-1}_\sA(U, \bbLog^N(g))$ is characterized as the class
	which maps to $1$ through the isomorphism
	\[
		 H^{2g-1}_\sA(U, \bbLog^N(g)) \cong  \Hom_{\MHS_\bbR}(\bbR(0), H^{2g-1}(U, \bbLog^N(g)) )
		 \overset{\res_1}{\cong} \Hom_{\MHS_\bbR}(\bbR(0), H^{0}(Z, \bbLog^N) )
		 \cong \bbR. 
	\]
\end{remark}

In what follows, let
\begin{align*}
	R^{\bullet}(U, \bbLog(g))&\coloneqq \varprojlim_N R^{\bullet}(U, \bbLog^N(g)),  &
	R^{\bullet}_\sA(U, \bbLog(g))&\coloneqq \varprojlim_N R_\sA^{\bullet}(U, \bbLog^N(g)).
\end{align*}
The filtrations $W_\bullet$ and $F^\bullet$ on $R^{\bullet}(U, \bbLog^N(g))$ induce
filtrations $W_\bullet$ and $F^\bullet$ on $R^{\bullet}(U, \bbLog(g))$.

\begin{proposition}\label{prop: Klingon}
	For any integer $m\geq 0$, we have
	\begin{align*}
		H^m(R^{\bullet}(U, \bbLog(g)))&\cong H^m(U,\bbLog(g)).
	\end{align*}
	Moreover, for any integer $n\in\bbZ$, we have
	\[
		H^m\biggl( \varprojlim_N W_n R^{\bullet}(U, \bbLog^N(g))\biggr) \cong \varprojlim_NH^m\bigl(W_nR^{\bullet}(U, \bbLog^N(g))\bigr).
	\]	
\end{proposition}

\begin{proof}	
	We start by proving the first isomorphism, that is
	\[
		H^m\biggl( \varprojlim_N R^{\bullet}(U, \bbLog^N(g))\biggr) \cong \varprojlim_NH^m\bigl(R^{\bullet}(U, \bbLog^N(g))\bigr).
	\]
	By \cite{KS94}*{Proposition 1.12.4}, it is sufficient to check that the systems $\bigl(R^m(U,\bbLog^N(g))\bigr)_N$
	and $\bigl(H^{m}(R^\bullet(U,\bbLog^N(g)))\bigr)_N$  satisfy the Mittag-Leffler condition for any $m\in\bbZ$.
	For $\bigl(R^m(U,\bbLog^N(g))\bigr)_N$, recall the definition of $\sLog^N(D)$ given in Definition \ref{def: sLog}.
	As an $\sE_{\ol X,\bbR}$-module, we have a splitting
	\begin{equation}\label{eq: Klingon}
		\sLog^{N+1}(D) = \sLog^{N}(D) \oplus W_{-2N-2}\sLog^{N+1}(D),
	\end{equation}
	and the projection $\sLog^{N+1}(D) \rightarrow  \sLog^{N}(D)$ is simply the projection to the first component.
	This implies that
	\[
		\sLog^{N+1}(D)(g)\otimes \sE^p_{\ol X,\bbR}(D^q)\rightarrow \sLog^N(D)(g)\otimes \sE^p_{\ol X,\bbR}(D^q)
	\]
	is split surjective, hence the induced morphism $R^m(U,\bbLog^{N+1}(g))\rightarrow R^m(U,\bbLog^{N}(g))$ is also split surjective.
	The Mittag-Leffler condition for $\bigl(H^{m}(R^\bullet(U,\bbLog^N(g)))\bigr)_N$ follows from Lemmas \ref{lemma: Log on X} and \ref{lemma: Log on U}.

	We prove the second isomorphism in a similar way.	
	Since the splitting \eqref{eq: Klingon} is compatible with the filtration $W_\bullet$,
	the system $\bigl(W_n R^m(U,\bbLog^N(g))\bigr)_N$ satisfies the Mittag-Leffler condition.
	On the other hand, the Mittag-Leffler condition for $\bigl(H^m(R^\bullet(U,\bbLog^N(g)))\bigr)_N$
	shown above implies that for $\bigl(W_n H^m(R^\bullet(U,\bbLog^N(g)))\bigr)_N$, since the morphisms of mixed $\bbR$-Hodge
	structures are strictly compatible with the filtrations.
	By \eqref{eq: weight2}, the latter system is equal to $\bigl(H^m(W_n R^\bullet(U,\bbLog^N(g)))\bigr)_N$,
	hence we obtain our assertion.
\end{proof}

In the following subsections, we will construct an explicit cocycle giving the polylogarithm class.
More precisely, we will define a triple
\[
	(\alpha,\eta,\xi)\in  W_0R^{2g-2}(U, \bbLog_\bbC(g)) \oplus W_0R^{2g-1}(U, \bbLog(g))
	 \oplus  (F^0 \cap W_0)R^{2g-1}(U, 	\bbLog_\bbC(g)),
\]
which satisfies:
\begin{enumerate}
	\item The cocycle conditions $d\alpha=\eta-\xi$, $d\eta=0$ and $d\xi=0$.
	\item $\xi=(\xi_N)$, where $\xi_N$ are the differential forms of \eqref{eq: differential form} giving the geometric polylogarithm class.
\end{enumerate}

The first condition insures that $(\alpha,\eta,\xi)$ defines a class in $H^{2g-1}_\sA(U, \bbLog(g))$,
and the second implies that this class gives the absolute polylogarithm class $\pol$.

%
%
\subsection{Polylogarithm Function and the Case $\boldsymbol{g=1}$}
%
%

We first review the definition of the classical polylogarithm function.
The polylogarithm function, defined by the power series
\[
	\Li_k(t):=\sum_{n=1}^\infty \frac{t^n}{n^k}, \qquad (t\in\bbC, |t|<1, k\in\bbN),
\]
may be extended analytically to $\bbC\setminus[1,\infty)$ using the integration
\[
	\Li_{k+1}(t) = \int_{0}^t  \Li_k(u) \frac{du}{u}
\]
for any integer $k\geq 0$.  

For $m\geq0$ and $|t| < 1$, let
\begin{align}\label{eq: L}
	L_{m+1}(t):= \sum_{k=0}^{m} \frac{(-\log|t|)^{m-k}}{(m-k)!} \Li_{k+1}(t).
\end{align}

In \cite{Zag90}*{\S1}, Zagier defined the Bloch-Wigner-Ramakrishnan polylogarithm functions $D^\circ_m(t)$ by
\[
	D^\circ_m(t):=
	\begin{cases}
		\displaystyle\Im(L_m(t)) &  \text{($m$: even)},\\
		\displaystyle\Re(L_m(t)) +\frac{(\log|t|)^m}{(2m)!}&  \text{($m$: odd)}.
	\end{cases}
\]
In this article, we take a slightly different normalization as follows:
\begin{definition}\label{definition: BWR}
	For any $m\in\bbN$, we define the \emph{Bloch-Wigner-Ramakrishnan polylogarithm function} $D_m(t)$ by
	\begin{equation}\label{eq: Ramakrishnan}
		D_m(t):= \Im(i^m L_m(t)).
	\end{equation}
\end{definition}

The relations between the normalizations are
\[
	D_m(t)=\Im(i^m L_m(t)) = (-1)^n D^\circ_m(t)
\]
if $m=2n$ is even, and
\[
	D_m(t)=	\Im(i^m L_m(t))= (-1)^n \Re(L_m(t))=  (-1)^n D^\circ_m(t) - (-1)^n \frac{(\log|t|)^m}{(2m)!}
\]
if $m=2n+1$ is odd.
Then the functions $D_m(t)$ satisfy the following functional equations.
\begin{proposition}[\cite{Zag90}*{Proposition 1}]\label{prop: Ramakrishnan}
	The function $D_m(t)$ of \eqref{eq: Ramakrishnan}  can be continued to a single-valued real analytic function on $U:=\bbP^1\setminus\{0,1,\infty\}$.
	In particular, it is a function in $\sE_{U,\bbR}$.  The function $D_m(t)$ satisfies the functional equation
	\begin{equation}\label{eq: functional equation}
		D_m(1/t)=
		\begin{cases}
			D_m(t) &  \text{($m$: even)},\\
			D_m(t) + (\log|t|)^m/m!  &  \text{($m$: odd)}.
		\end{cases}
	\end{equation}
\end{proposition}

We will use the functions $D_m(t)$ given above to explicitly describe the polylogarithm class.
In this subsection, we consider the case $g=1$.
Let $\ol X= \bbP^1$, $D= \{0,\infty\}$ and $Z=\{1\}$.

\begin{theorem}[Case $g=1$]\label{thm: g=1}
	The polylogarithm class $\pol$ is given by the triple $(\alpha, \eta, \xi)$, where 
	\begin{align*}
		\xi&= \ul\omega^0(1) \otimes \frac{dt}{t-1} \in (F^0 \cap W_0)R^1(U,\bbLog_\bbC(1)),\\
		\eta&\coloneqq \Re(\xi) = (\xi + \ol\xi)/2  \in W_0 R^1(U,\bbLog(1)),
	\end{align*}
	and
	\[
		\alpha\coloneqq \sum_{m=0}^\infty D_{m+1}(t) \ul e^m(1) \in W_0R^0(U,\bbLog(1)).
	\]
\end{theorem}

\begin{proof}
	Let us consider 
	\[
		\wt\alpha\coloneqq\sum_{k=0}^\infty (-1)^{k+1} \Li_{k+1}(t)\ul\omega^{k}(1)
	\]
	which is a multi-valued section of $\cLog(1)$ on $U$.
	Then we have
	\begin{align*}
		\nabla(\wt\alpha)& = \sum_{k=0}^\infty (-1)^{k+1} \ul\omega^{k}(1) \otimes d\Li_{k+1}(t)+  \sum_{k=0}^\infty (-1)^{k+1} \Li_{k+1}(t)\ul\omega^{k+1}(1) \otimes \frac{dt}{t}\\
		&= \sum_{k=0}^\infty (-1)^{k+1}  \Li_{k}(t) \ul\omega^{k}(1) \otimes \frac{dt}{t}+  \sum_{k=1}^\infty (-1)^{k} \Li_{k}(t)\ul\omega^{k}(1) \otimes \frac{dt}{t}
	= -\Li_0(t)\ul\omega^0(1)\otimes\frac{dt}{t} = \xi.
	\end{align*}
	The section $\wt\alpha$ only gives a section of
	\[	
		\sLog(1)\otimes\sE_{\ol X}(D)\coloneqq\varprojlim_N\left(\sLog^N(1)\otimes\sE_{\ol X}(D)\right)
	\]
	on $\bbC\setminus[1,\infty)$
	and does not extend to a global section of $\ol X$, since the functions $\Li_k(t)$ extend to multi-valued functions
	and do not extend to single-valued functions on $\bbC\setminus\{1\}$.  Note that we have
	$
		\nabla(\Re(\wt\alpha)) = \Re(\xi) = \eta,
	$
	hence if we let
	\[
		\alpha\coloneqq\Re(\wt\alpha) - \wt\alpha = -i\,\Im(\wt\alpha),
	\]
	then we have 
	\begin{equation}\label{eq: alpha differential}
		\nabla(\alpha) = \eta - \xi =  -i\,\Im(\xi).
	\end{equation}
	By \eqref{eq: base change two}, we have
	\begin{align*}
		\wt\alpha& = \sum_{k=0}^\infty(-1)^{k+1}\Li_{k+1}(t) (-i)^k \exp(\log|t|\omega)\cdot\ul e^k(1)
		= \sum_{k=0}^\infty\sum_{n=0}^\infty  (-1)^{k+1} (-i)^{n+k}\frac{(\log|t|)^n}{n!} \Li_{k+1}(t)  \ul e^{n+k}(1)\\
	&= - \sum_{m=0}^\infty  i^m \left(\sum_{k=0}^m   \frac{(-\log|t|)^{m-k}}{(m-k)!}\Li_{k+1}(t)\right) \ul e^{m}(1)
	=\sum_{m=0}^\infty i^{m+1}L_{m+1}(t)  \cdot i\ul e^m(1).
	\end{align*}
	Since the basis $\ul e^m(1)$ is purely imaginary,
	we see that $\alpha$ is given by
	\begin{equation*}
		\alpha= -i\,\Im(\wt\alpha) 
		= -i \sum_{m=0}^\infty \Im(i^{m+1}L_{m+1}(t)) \cdot i\ul e^m(1)
		= \sum_{m=0}^\infty D_{m+1}(t) \ul e^m(1),
	\end{equation*}
	which by Proposition \ref{prop: Ramakrishnan} is an element in $W_0\Gamma(U, \sLog(1) \otimes \sE_{U,\bbR})$.  
	
	It remains to show that $\alpha$ defines an element of $W_0 R^0(U,\bbLog(1))=W_0
	\Gamma\bigl(\ol X,\sLog(D)(1)\otimes\sE_{\ol X,\bbR}(D\cup Z)\bigr)$,
	From the fact that the class $[\xi]$ is real, we have $[\xi]=[\eta]$.
	Hence by Proposition \ref{prop: Klingon}, there exists $\alpha'\in W_0 R^0(U,\bbLog(1))$ such that $\nabla(\alpha')=\eta-\xi$.
	Since
	\[	
		\nabla(\alpha- \alpha'|_U) = \nabla(\alpha) - \nabla(\alpha'|_U)=0,
	\]
	the section $w:= \alpha-\alpha'|_U 
	\in \Gamma(U, \sLog(1)\otimes\sE_{U,\bbR})$
	defines a class in $\bbH^0(U,  \sLog(1) \otimes \sE^\bullet_{U,\bbR}) = H^0(U, \bbLog(1))$.  
	By Proposition \ref{prop: Log on U}, we have $H^0(U,\bbLog(1))=0$,
	hence this shows that $w=0$, which implies that $\alpha'|_U = \alpha$.
	This implies that $\alpha$ extends to an element of $W_0 R^0(U,\bbLog(1))$
	as desired.
\end{proof}

Using this result, we may recover the well-known result concerning the specialization of the polylogarithm class to torsion points of $\bbG_{\mathrm{m}}$.
Let $d>1$ be an integer and let $\zeta$ be a primitive $d$-th root of unity.
Let $\bbR(n) = (V_\bbR, W_\bullet, F^\bullet)$ with $V_\bbR:= \bbR \ul u^n$, $W_0 V_\bbR = V_\bbR$ if $n \geq 0$ and $F^0 V_\bbC = 0$ if $n>0$
be the Tate object of Example \ref{example: Tate}.
The inclusion $i_\zeta: \Spec \bbC \rightarrow U$
defines the restriction map 
\[
	i^*_\zeta: H^1_\sA(U, \bbLog(1)) \rightarrow H^1_\sA(\Spec \bbC, i^*_\zeta \bbLog(1)) \cong \prod_{k=0}^\infty H^1_\sA(\Spec \bbC, \bbR(k+1)),
\]
where the last equality follows from the splitting principle \eqref{eq: splitting}.
For any mixed $\bbR$-Hodge structure $V = (V_\bbR, W_\bullet, F^\bullet)$, we have by definition
\[
	H^m_\sA(\Spec \bbC, V):= \Ext^m_{\MHS_\bbR}(\bbR(0), V) = H^m \left( \Cone\left(W_0 V_\bbR \oplus (F^0 \cap W_0) V_\bbC  \rightarrow W_0V_\bbC \right)[-1]\right).
\]
For $\bbR(n) = (V_\bbR, W_\bullet, F^\bullet)$, if $n>0$ then we have
\[
	H^1_\sA(\Spec \bbC, \bbR(n)) \cong W_0 V_\bbC / W_0 V_\bbR = \bbC \ul\omega^n/\bbR \ul u^n \cong \bbC/(2\pi i)^n \bbR.
\]
Our calculation gives the following:

\begin{corollary}\label{corollary: main}
	Let $d>1$ be an integer and let $\zeta$ be a primitive $d$-th root of unity.  Then we have
	\[
		i^*_\zeta \pol  = ( (-1)^k\Li_{k+1}(\zeta))_{k=0}^\infty \in \prod_{k=0}^\infty \bbC/(2\pi i)^{k+1} \bbR.
	\]
\end{corollary}

\begin{proof}
	By Theorem \ref{thm: g=1}, the class $i^*_\zeta \pol$ is represented by the triple 
	\begin{multline*}
		(i^*_\zeta \alpha, i^*_\zeta \eta, i^*_\zeta\xi) \in W_0 R^0(\Spec\bbC, i^*_\zeta\bbLog_\bbC(1)) \oplus
		W_0 R^1(\Spec\bbC, i^*_\zeta\bbLog_\bbR(1))\\ \oplus (F^0\cap W_0) R^1(\Spec\bbC, i^*_\zeta\bbLog_\bbC(1)).
	\end{multline*}
	However, since the dimension of $\Spec\bbC$ is zero, we have 
	\[
		R^1(\Spec\bbC, i^*_\zeta\bbLog_\bbR(1))=R^1(\Spec\bbC, i^*_\zeta\bbLog_\bbC(1))=0.
	\]
	Hence $i^*_\zeta \pol$ is represented by the section
	\[
		i^*_\zeta \alpha = \sum_{m=0}^\infty D_{m+1}(\zeta) \ul e^m(1) 
		\in W_0 R^0(\Spec\bbC, i^*_\zeta\bbLog_\bbC(1)) := \Gamma(\Spec\bbC, i^*_\zeta\bbLog_\bbC(1))
		= \prod_{m=0}^\infty \bbC \ul e^m(1).
	\]
	Note that since $\log|\zeta|=0$, we have $L_{m+1}(\zeta) = \Li_{m+1}(\zeta)$ for any $m\geq0$.
	By definition,
	\begin{align*}
		D_{m+1}(\zeta) \ul e^m(1)&= \Im(i^{m+1} L_{m+1}(\zeta)) \ul e^m(1)= \Im(i^{m+1} \Li_{m+1}(\zeta)) \ul e^m(1)\\
		&= \Re(i^{m} \Li_{m+1}(\zeta)) \ul e^m(1).
	\end{align*}
	hence we have
	\[
		D_{m+1}(\zeta) \ul e^m(1) \equiv i^{m} \Li_{m+1}(\zeta) \ul e^m(1)  \pmod{\bbR i\ul e^m(1)}.
	\]
	Noting that $i^{m} \ul e^m(1) = (-1)^m\ul\omega^{m+1}$ and $\bbR i\ul e^{m}(1) = \bbR\ul u^{m+1}$, we see that $i^*_\zeta \alpha$ coincides with the class of 
	\[
		\sum_{m=0}^\infty (-1)^m\Li_{m+1}(\zeta) \ul\omega^{m+1}
	\]
	in $\prod_{m=0}^\infty \bbC\ul\omega^{m+1}/\bbR\ul u^{m+1}$ as desired.
\end{proof}

%
%
\subsection{The Polylogarithm Class for the Case $\boldsymbol{g>1}$.}
%
%

We consider the case when $X= \bbG_{\mathrm{m}}^g$ for an integer $g>1$.
For $\mu\in[g]$, we let $\alpha_\mu$ and $\xi_\mu$ be the elements given by
\begin{align*}
	\alpha_\mu &:= \sum_{m=0}^\infty D_{m+1}(t_\mu) \ul e^{m_\mu}(1),   &   \xi_\mu &:= \ul\omega^{\boldsymbol{0}}(1) \otimes \frac{dt_\mu}{t_\mu-1},
\end{align*}	
where $m_\mu$ is the element in $\bbN^g$ with $m$ in the $\mu$-th component and $0$ in the other components.
Then the differential form $\xi$ of \eqref{eq: differential form} is given by
\[
 	\xi= \ \xi_1 \wedge \cdots \wedge\xi_g  \in (F^0\cap W_0) R^{2g-1}(U, \bbLog_\bbC(g)).
\]
We let
\begin{equation*}
	\eta\coloneqq\Re(\xi ) = (\xi + \ol \xi)/2 \in W_0 R^{2g-1}(U, \bbLog(g)).
\end{equation*}
Then we have $\eta - \xi = (\ol \xi - \xi)/2$.
If we let
\begin{equation*}
	\alpha\coloneqq\sum_{\mu=1}^g (-1)^{\mu-1} \alpha_\mu \xi_1 \wedge \cdots \wedge \xi_{\mu-1} \wedge \ol{\xi_{\mu+1}} \wedge \cdots \wedge \ol{\xi_g}
	\in W_0 R^{2g-2}(U, \bbLog_\bbC(g)),
\end{equation*}
then we have
\begin{align*}
	\nabla(\alpha)&= \sum_{\mu=1}^g (-1)^{\mu-1}  \frac{\ol \xi_\mu - \xi_\mu}{2} \wedge \xi_1 \wedge \cdots \wedge \xi_{\mu-1} \wedge \ol{\xi_{\mu+1}} \wedge \cdots \wedge \ol{\xi_g}\\
	&= \frac{1}{2} \sum_{\mu=1}^g  \left( \xi_1 \wedge \cdots \wedge \xi_{\mu-1}\wedge \ol{\xi_\mu} \wedge \ol{\xi_{\mu+1}} \wedge \cdots \wedge \ol{\xi_g}
	-  \xi_1 \wedge \cdots \wedge \xi_{\mu-1}\wedge \xi_\mu \wedge \ol{\xi_{\mu+1}} \wedge \cdots \wedge \ol{\xi_g} \right)\\
	&= \frac{1}{2} \left(\ol{\xi_1} \wedge \cdots \wedge\ol{\xi_g}  -   \xi_1\wedge\cdots\wedge\xi_g \right)  = \frac{1}{2} (\ol{\xi} - \xi) = \eta- \xi,
\end{align*}
where the calculation of $\nabla(\alpha_\mu)$ follows from \eqref{eq: alpha differential}.

This proves our main result.

\begin{theorem}[Case $g>1$]\label{main theorem}
	We let $\alpha$, $\eta$ and $\xi$ be as above.  Then the triple
	\[
		(\alpha, \eta, \xi) \in W_0 R^{2g-2}(U, \bbLog_\bbC(g)) \oplus W_0 R^{2g-1}(U, \bbLog(g)) \oplus  (F^0 \cap W_0)R^{2g-1}(U, \bbLog_\bbC(g))
	\]
	represents the polylogarithm class $\pol$ in $H^{2g-1}_\sA(U,\bbLog(g))$.
\end{theorem}

\end{document}